\documentclass[10pt,reqno]{amsart}
\usepackage{epsfig,amssymb,amsmath,version}
\usepackage{amssymb,version,graphicx,fancybox,mathrsfs,multirow}

\usepackage{url,hyperref}

\usepackage{subfigure}
\usepackage{color,xcolor}
\usepackage{cases}
\usepackage{mathtools}

\catcode`\@=11 \theoremstyle{plain}
\@addtoreset{equation}{section}   
\renewcommand{\theequation}{\arabic{section}.\arabic{equation}}
\@addtoreset{figure}{section}
\renewcommand\thefigure{\thesection.\@arabic\c@figure}
\renewcommand{\thefigure}{\arabic{section}.\arabic{figure}}
\newtheorem{thm}{\bf Theorem}

\newtheorem{cor}{\bf Corollary}

\newtheorem{lmm}{\bf Lemma}

\newenvironment{lemma}{\begin{lmm}}{\end{lmm}}
\theoremstyle{remark}
\newtheorem{rem}{\bf Remark}[section]
\theoremstyle{definition}
\newtheorem{defn}{\bf Definition}[section]

\numberwithin{table}{section}

\def \af {\alpha}
\def \bt {\beta}

\renewcommand \wedge \times

\def \dlx {D_{-}^{s}}

\def \drx {D_{+}^{s}}

\begin{document}
\bibliographystyle{plain}

\title[GJFs and Their Approximations to FDEs] {Generalized Jacobi Functions and Their Applications to Fractional Differential Equations}
\author{Sheng Chen${}^1$,\;\;  Jie Shen${}^{2,1}$\;\;  and\;\;   Li-Lian Wang${}^3$}

\thanks{${}^1$School of Mathematical Sciences, Xiamen University, Xiamen, Fujian 361005, P. R. China.}
\thanks{${}^2$Department of Mathematics, Purdue University, West Lafayette,
  IN 47907-1957, USA. J.S. is  partially  supported by NSF grant
        DMS-1217066 and AFOSR grant FA9550-11-1-0328.}
 \thanks{${}^3$Division of Mathematical Sciences, School of Physical
and Mathematical Sciences, Nanyang Technological University,
637371, Singapore. The research of this author is partially supported by Singapore MOE AcRF Tier 1 Grant (RG 15/12), Singapore MOE AcRF Tier 2 Grant (MOE 2013-T2-1-095, ARC 44/13) and Singapore A$^\ast$STAR-SERC-PSF Grant (122-PSF-007).}

\begin{abstract} In this paper,  we consider spectral approximation of fractional differential equations (FDEs). A main ingredient of our approach is to define a new class of generalized Jacobi functions (GJFs),  which is intrinsically related to fractional calculus, and  can serve as natural basis functions  for properly designed
spectral methods for FDEs.  We establish spectral approximation results for these GJFs in weighted Sobolev spaces involving fractional derivatives.
We construct efficient GJF-Petrov-Galerkin methods for a class of prototypical fractional initial value problems (FIVPs) and fractional boundary value problems  (FBVPs)  of general order, and show that with an appropriate choice of the parameters in GJFs, the resulted linear systems can be sparse and well-conditioned.   Moreover,
 we derive  error estimates with convergence rate  only depending on the smoothness of data, so truly spectral accuracy can be attained if the data are smooth enough.
 The idea  and results presented in this paper will be useful to deal with more general FDEs associated with Riemann-Liouville or Caputo fractional derivatives.
\end{abstract}

 \maketitle
\section{Introduction}
Fractional  differential equations  appear in the investigation of transport dynamics in complex systems which are governed by the anomalous diffusion and non-exponential relaxation patterns. Related equations of importance are the space/time fractional diffusion equations, the fractional advection-diffusion equations  for anomalous diffusion with sources and sinks,  the fractional FokkerÐPlanck equations  for anomalous diffusion in an external field, and among others.
Progress in the last two decades has demonstrated that  many phenomena in various fields of science, mathematics, engineering, bioengineering, and economics are more accurately described by involving  fractional derivatives. Nowadays,   FDEs are emerging as a new powerful tool for modeling many difficult type of complex systems, i.e., systems with overlapping microscopic and macroscopic scales or systems with long-range time memory and long-range spatial interactions (see, e.g.,  \cite{Pod99,Met.K00,Kil.S06,Diet10,DGLZ13b} and the references therein).

There has been a growing interest in the last decades in developing numerical methods for solving FDEs, and a large volume  of literature  is available on this subject.
Generally speaking, two  main difficulties for dealing with FDEs are
\begin{itemize}
 \item[(i)]  fractional derivatives are  non-local operators;
 \item[(ii)] fractional derivatives involve singular kernel/weight functions,  and the solutions of FDEs are usually singular near the boundaries.
 \end{itemize}

 Most of the existing numerical methods for FDEs are based on finite difference/finite element methods (cf. \cite{Mee.T04,LAT04,Sun.W06,MST06,EHR07,Erv.R07,Tad.M07,JLZ13,ZLLT13} and the references therein) which lack the capability to effectively deal with the aforementioned  difficulties,  as they are based on ``local" operations, and are  not well-suited  for problems with singular kernels/weights.
 In particular,
 due to the non-local nature of the fractional derivatives, they all lead to full and dense matrices which are expensive to calculate and to invert.
 Recently, some interesting ideas  have been
proposed to overcome these difficulties.  For instance, Wang and  Basu \cite{Wan.B12}  proposed a fast finite-difference method by carefully analyzing the structure of the  coefficient matrices of the resulted linear systems, and delicately decomposing them into a combination of sparse and structured dense matrices.

There exist also  limited but very promising efforts in developing spectral methods for solving FDEs  (see, e.g.,  \cite{Li.X09,Li.X10,LZL12,zayernouri2013fractional,Zayernouri.K14}).
The  spectral method  appears to be  a natural approach,  since it is  global,  which should be  better suited for non-local problems.
Most notably,
 Zayernouri and Karniadakis \cite{zayernouri2013fractional} proposed to use polyfractomials  as basis functions,  which are  eigenfunctions of a fractional Sturm-Liouville operator,  and   result in  sparse matrices for some simple model equations.  Preliminary results in \cite{zayernouri2013fractional} showed that this new approach could lead to several orders of magnitude saving in CPU and memory for some  model FDEs.  However, there is no error analysis available for the
 approximation properties  of  polyfractomials, and the algorithms therein 
 do not necessarily lead to spectral convergence for problems with smooth data  but non-smooth solution which is typical for FDEs.

 The second difficulty is largely  ignored in the literature.  Typically,  the solution and data of a FDE are not in the same type of Sobolev spaces,  which is in distinctive contrast with  usual DEs.  Consequently, they should be approximated by different tools, and the error estimates  should be measured in norms of different types of spaces.
 Indeed, given smooth data, the solution of a FDE only has limited regularity  in the usual Sobolev spaces.
 However, existing error estimates for FDEs, either finite differences, finite elements or spectral methods,  are all based on the usual approach, namely,  the errors are performed in the framework of  usual Sobolev spaces.  Hence, it is not surprising to see  that
 most existing methods and the related error estimates only lead to  poor convergence rate  for
 typical FDEs, unless one   manufactures a smooth exact solution,  directly uses a polynomial-based method, and then carefully deals with the singular data.
%

 The  purpose of this paper is to  develop and analyze efficient spectral methods which can effectively address the above two issues for a class of prototypical FDEs.
 The main strategies and contributions are highlighted as follows.
 \begin{itemize}
 \item  We introduce a new class of GJFs with two parameters, which    can be tuned to match singularity of the underlying solution, and simultaneously produce sparse
 linear systems.  More importantly, such GJFs enjoy attractive  fractional calculus properties and remarkable approximability to functions with singular behaviour at boundaries.
 \item We derive optimal approximation results for these GJFs in suitably weighted spaces involving fractional derivatives, and  obtain error estimates for the proposed
 GJF-Petrov-Galerkin approaches  with convergence rate only depending on smoothness of the data (characterised by  usual Sobolev norms).  Thus, truly spectral accuracy can be achieved for some model FDEs with sufficient smooth data.
 \item We point out  that the  GJFs,  including generalised Jacobi polynomials (GJPs) as special cases,  have been first   introduced in  \cite{Guo.SW06,Guo.SW09} for solutions of usual BVPs. Here,  we  modify the original definition, especially  the range of the parameters,  which  opens up new applications in solving  FDEs.
 We also remark that GJFs with parameters in  $(0,1)$  have direct bearing on the Jacobi  polyfractomials
 in \cite{zayernouri2013fractional}.  The major difference from  these relevant existing ones lies in that
 the new GJFs are built upon   Jacobi polynomials {\em with real parameters.}    This is essential for both algorithm development and error   analysis.
 \end{itemize}
 While we shall only consider some prototypical FIVPs and FBVPs of general order,  we  position  this work as the first but important step towards developing efficient spectral  methods for more complicated FDEs involving Riemann-Liouville or Caputo fractional derivatives.


The paper is organized as follows. In the next section, we make necessary preparations by recalling  basic properties of Jacobi polynomials with real parameters, and introducing the important Bateman fractional integral formula. In Section \ref{sect:GJFs}, we define the GJFs and derive their essential  properties, particularly, including  fractional calculus properties.  In Section \ref{sect:appr}, we establish the approximation results for  these GJFs. In Section \ref{sect:appl}, we construct efficient GJF-Petrov-Galerkin methods for a class of
prototypical FDEs, conduct error analysis and present ample supporting numerical results.
In the final section, we extend the most important Riemann-Liouville fractional derivative formulas to the Caputo fractional derivatives, and conclude the paper with a few remarks.

\section{Preliminaries}

In this section, we review   basics  of fractional integrals/derivatives,
and recall   relevant properties of the Jacobi polynomials with real parameters.
In particular, we introduce the  Bateman fractional integral formula, which plays a very  important role
in the forthcoming algorithm development and analysis.

 \subsection{Fractional integrals and derivatives} Let $\mathbb N$ and $\mathbb R$ be  the set of positive integers and real numbers, respectively. Denote
 \begin{equation}\label{mathNR}
{\mathbb N}_0:=\{0\}\cup {\mathbb N}, \quad {\mathbb R}^+:= \{a\in {\mathbb R}: a> 0\}, \quad {\mathbb R}^+_0:=\{0\}\cup {\mathbb R}^+.
\end{equation}
We first recall the definitions of the fractional integrals and fractional derivatives in the sense of Riemann-Liouville and Caputo (see,  e.g., \cite{Pod99,Diet10}). To fix the idea, we restrict our attentions to the interval $(-1,1)$.
It is clear that all formulas and properties  can be  formulated on  a general interval $(a,b)$.
\begin{defn}[{\bf Fractional integrals and derivatives}]\label{RLFDdefn}
{\em For  $\rho\in {\mathbb R}^+,$  the left and right fractional integrals are respectively defined  as
 \begin{equation}\label{leftintRL}
 \begin{split}
   &I_{-}^\rho v(x)=\frac 1 {\Gamma(\rho)}\int_{-1}^x \frac{v(y)}{(x-y)^{1-\rho}} dy,\; x>-1;\;\;   I_{+}^\rho v(x)=\frac {1}  {\Gamma(\rho)}\int_{x}^1 \frac{v(y)}{(y-x)^{1-\rho}} dy,\; x<1,
   \end{split}
\end{equation}
where $\Gamma(\cdot)$ is the usual Gamma function.

For  $s\in [k-1, k)$ with $k\in {\mathbb N},$  the left-sided  Riemann-Liouville fractional derivative {\rm(LRLFD)} of order $s$ is defined by
\begin{equation}\label{leftRLdefn}
   \dlx v(x)=\frac 1{\Gamma(k-s)}\frac{d^k}{dx^k}\int_{-1}^x \frac{v(y)}{(x-y)^{s-k+1}} dy,\;\;\;  x\in \Lambda:=(-1,1),
\end{equation}
and the right-sided Riemann-Liouville fractional derivative  {\rm(RRLFD)} of order $s$ is defined  by
\begin{equation}\label{rightRLdefn}
   \drx v(x)=\frac {(-1)^k}{\Gamma(k-s)}\frac{d^k}{dx^k}\int_{x}^1 \frac{v(y)}{(y-x)^{s-k+1}} dy, \;\;\;  x\in \Lambda.
\end{equation}

For  $s\in [k-1, k)$ with $k\in {\mathbb N},$  the left-sided Caputo fractional derivatives {\rm(LCFD)}  of order $s$ is  defined by
\begin{equation}\label{left-fra-der-c}
    {^C}\hspace*{-2pt}D_-^sv(x):=\frac{1}{\Gamma(k-s)}\int_{-1}^x\frac{v^{(k)}(y)}{(x-y)^{s-k+1}}dy, \;\;\;  x\in \Lambda,
\end{equation}
and the right-sided Caputo fractional derivatives {\rm(RCFD)}  of order $s$ is  defined by
\begin{equation}\label{right-fra-der-c}
    {^C}\hspace*{-2pt}D_+^sv(x):=\frac{(-1)^{k}}{\Gamma(k-s)}\int_{x}^1\frac{v^{(k)}(y)}{(y-x)^{s-k+1}}dy, \;\;\;  x\in \Lambda.
\end{equation}
}
\end{defn}
 It is clear that  for any   $k\in {\mathbb N}_0,$
\begin{equation}\label{2rela}
D_-^{k}=D^{k},\quad D_+^{k}=(-1)^{k} D^{k},\quad {\rm where}\;\;\;  D^k:={d^k}/{dx^k}.
\end{equation}
Thus,  we can define the RLFD as
\begin{equation}\label{ImportRela}
   \dlx v(x)=D^k\, I_{-}^{k-s} v(x), \quad    \drx v(x)=(-1)^k D^k \, I_{+}^{k-s} v(x). 
\end{equation}
According to \cite[Thm. 2.14]{Diet10},  we have that for any absolutely integrable function $v,$ and real $s\ge 0,$
\begin{equation}\label{rulesa}
   D_\pm^s\, I_\pm^s v(x) = v(x),\quad \text{a.e. in} \;\; \Lambda.
 \end{equation}
%

 The following lemma shows the relationship between the Riemann-Liouville  and Caputo fractional derivatives (see,  e.g., \cite[Ch. 2]{Pod99}).
\begin{lmm}\label{relation-RL-Caputo}
For  $s\in [k-1, k)$ with $k\in {\mathbb N},$  we have  
\begin{subequations}\label{relaCaputo}
\begin{gather}
  D_-^sv(x) ={}^C\hspace*{-2pt}D_-^s v(x)+ \sum_{j=0}^{k-1} \frac{v^{(j)}(-1)}{\Gamma(1+j-s)}(1+x)^{j-s};
  \label{left-relation-RL-Caputo} \\
  D_+^s v(x)={^C}\hspace*{-2pt}D_+^s v(x) + \sum_{j=0}^{k-1} \frac{(-1)^{j} v^{(j)}(1)}{\Gamma(1+j-s)}(1-x)^{j-s}. \label{right-relation-RL-Caputo}
\end{gather}
\end{subequations}
\end{lmm}
\begin{rem}\label{Gammarem}  {In the above, the Gamma function with negative, non-integer  argument should be understood by  the Euler  reflection formula {\rm(}cf. \cite{Abr.I64}{\rm)}:
$$
\Gamma(1+j-s)=\frac{\pi}{\sin (\pi (1+j-s))}\frac{1}{\Gamma(s-j)},\;\;\;  s\in (k-1,k),\;\; 1\le j\le k-2.
$$
Note that if $s=k-1,$ then $\Gamma(1+j-s)=\infty$ for all $0\le j\le k-2,$ so the summations in the above reduce  to $v^{(k-1)}(\pm 1),$ respectively.
\qed }
\end{rem}

\begin{rem}\label{RLCa}
We observe immediately  from \eqref{relaCaputo} that  for  $s\in [k-1, k)$ with $k\in {\mathbb N},$
\begin{equation}\label{RL=Ca}
 D_\pm^s v(x)={}^C\hspace*{-2pt}D_\pm^s v(x), \quad\text{if}\;\; v^{(j)}(\pm 1)=0,\;\;  0\le j\le k-1.
\end{equation}
\end{rem}

The rule of factional integration by parts (see, e.g., \cite{FSLP}) will also be used subsequently.
\begin{lmm}\label{integration}  For  $s\in [k-1, k)$ with $k\in {\mathbb N},$  we have
\begin{subequations}\label{byparts}
\begin{gather}
\big(D_-^su, v\big)  =\big(u, {}^C\hspace*{-2pt}D_+^s v\big) + \sum_{j=0}^{k-1} (-1)^{j} {v^{(j)}(x)}D^{k-j-1} I_-^{k-s} u (x)\Big|_{x=-1}^{x=1};
  \label{left-byparts} \\
\big(D_+^su, v\big)  =\big(u, {}^C\hspace*{-2pt}D_-^s v\big) + \sum_{j=0}^{k-1} (-1)^{k-j} {v^{(j)}(x)}D^{k-j-1} I_+^{k-s} u (x)\Big|_{x=-1}^{x=1},
  \label{right-byparts}
\end{gather}
\end{subequations}
where  $(\cdot,\cdot)$ is the   $L^2$-inner product.
\end{lmm}

\subsection{Jacobi polynomials with real parameters}  Much of our discussion later   will make use of  Jacobi polynomials with real parameters. Below,
we review their relevant properties.

Recall   the hypergeometric function (cf.  \cite{Abr.I64}):
\begin{equation}\label{hyperboscs}
{}_2F_1(a,b;c; x)=\sum_{j=0}^\infty \frac{(a)_j (b)_j}{(c)_j}\frac{x^j}{j!},\quad |x|<1,\;\;\; a,b,c\in {\mathbb R},\; -c\not \in {\mathbb N}_0,
\end{equation}
where   the rising factorial in the Pochhammer symbol, for $a\in {\mathbb R}$ and  $j\in {\mathbb N}_0,$ is defined by:
\begin{equation}\label{anotation}
(a)_0=1; \;\;\;  
(a)_j:=a(a+1)\cdots (a+j-1)=\frac{\Gamma(a+j)}{\Gamma(a)},\;\; {\rm for}\;\; j\ge 1.
\end{equation}
   If $a$ or $b$ is a negative integer, then it reduces to a  polynomial.

The classical Jacobi polynomials are defined for  parameters $\alpha,\beta>-1$.
  The Jacobi polynomials  can also be defined for  $\alpha\le -1$ and/or $\beta\le -1$ as in  Szeg\"o \cite[(4.21.2)]{szeg75}:
\begin{equation}\label{Jacobidefn0}
\begin{split}
P_n^{(\alpha,\beta)}(x)&=\frac{(\alpha+1)_n}{n!}{}_2F_1\Big(-n, n+\alpha+\beta+1;\alpha+1;\frac{1-x} 2\Big)\\
&=\frac {(\alpha+1)_n} {n!}+   \sum_{j=1}^{n-1} \frac{(n+\alpha+\beta+1)_j (\alpha+j+1)\cdots (\alpha+n)}{j!(n-j)!} \Big(\frac{x-1} 2\Big)^j\\
&\quad  + \frac {(n+\alpha+\beta+1)_n} {n!}  \Big(\frac{x-1} 2\Big)^n,\quad   n\ge 1,
\end{split}
\end{equation}
and  $P_0^{(\alpha,\beta)}(x)\equiv 1.$  Note that
{\em $P_n^{(\alpha,\beta)}(x)$  is always a polynomial in $x$ for all  $\alpha,\beta\in {\mathbb R}.$ }

Many properties of the classical  Jacobi polynomial (with    $\alpha,\beta>-1$)  can be extended to the general case (with  $\alpha,\beta\in {\mathbb R}$), see \cite[P. 62-67]{szeg75}.
In particular, there hold
 \begin{equation}\label{parity}
 P_n^{(\alpha,\beta)}(x)=(-1)^n P_n^{(\beta,\alpha)}(-x); \quad P_n^{(\alpha,\beta)}(1)=\frac{(\alpha+1)_n}{n!}.
 \end{equation}
Thus, we have the alternative representation:
\begin{equation}\label{Jacobidefn1}
\begin{split}
P_n^{(\alpha,\beta)}(x)&=(-1)^n \frac{(\beta+1)_n}{n!}{}_2F_1\Big(-n, n+\alpha+\beta+1;\beta+1;\frac{1+x} 2\Big),\;\; n\ge 1.
\end{split}
\end{equation}

  Since the leading coefficient of  $ P_n^{(\alpha,\beta)}(x)$ is $ {(n+\alpha+\beta+1)_n} /{(2^n n!)}$ (see  \eqref{Jacobidefn0}),  its  degree  is less than $n,$ when  $n+\alpha+\beta \in \{-1,\cdots, -n\}$ {\rm(}i.e.,  $(n+\alpha+\beta+1)_n=0).$   We also refer to  \cite[(4.22.3)]{szeg75} for details of the reduction.
%
 Throughout this paper, we assume that
\begin{equation}\label{conditions}
-(n+\alpha+\beta)\not \in {\mathbb N},\quad \forall\, n\ge 1,
\end{equation}
so  {\em $P_n^{(\alpha,\beta)}(x)$ is always a polynomial of degree $n.$}  Under the condition \eqref{conditions},
the  Jacobi polynomials defined by \eqref{Jacobidefn0}  can be computed by  the same three-term recurrence
 relation as the classical Jacobi polynomials:
\begin{equation}\label{ja3term}
\begin{split}
&P_{n+1}^{(\alpha,\beta)}(x)=\big(a_{n}^{\alpha,\beta}x-b_{n}^{\alpha,\beta}\big)P_{n}^{(\alpha,\beta)}(x)
-c_{n}^{\alpha,\beta}P_{n-1}^{(\alpha,\beta)}(x),\;\; n\ge 1,\\
&P_0^{(\alpha,\beta)}(x)=1,\;\;\; P_1^{(\alpha,\beta)}(x)=\frac 1
2(\alpha+\beta+2)x+\frac 1 2(\alpha-\beta),\\
\end{split}
\end{equation}
where
\begin{subequations}\label{jacoefabc}
\begin{gather}
a_{n}^{\alpha,\beta}=\frac{(2n+\alpha+\beta+1)(2n+\alpha+\beta+2)}
{2(n+1)(n+\alpha+\beta+1)},\label{jacoefa}\\
b_{n}^{\alpha,\beta}=\frac{(\beta^2-\alpha^2)(2n+\alpha+\beta+1)}
{2(n+1)(n+\alpha+\beta+1)(2n+\alpha+\beta)},\label{jacoefb}\\
c_{n}^{\alpha,\beta}=\frac{(n+\alpha)(n+\beta)(2n+\alpha+\beta+2)}
{(n+1)(n+\alpha+\beta+1)(2n+\alpha+\beta)}.\label{jacoefc}
\end{gather}
\end{subequations}

We particularly look at the Jacobi polynomials with one or both parameters being negative integers.
 If $\alpha=-l$ (with $l\in {\mathbb N}$),  $\beta\in {\mathbb R}$ and $n\ge l\ge 1,$  we
 have that (see  \cite[(4.22.2)]{szeg75})
\begin{equation}\label{alphaint}
P_n^{(-l,\beta)}(x)= d_{n}^{l,\beta} \Big(\frac{x-1} 2\Big)^l P_{n-l}^{(l,\beta)}(x), \; \; {\rm where}\;\;\;
d_{n}^{l,\beta}=\frac{(n-l)! (\beta+n-l+1)_l}{n!}.
\end{equation}
Similarly, 
for $\beta=-m,$ we find from \eqref{parity} and  \eqref{alphaint}  that
\begin{equation}\label{alphaint2}
P_n^{(\alpha,-m)}(x)= d_{n}^{m,\alpha} \Big(\frac{x+1} 2\Big)^m P_{n-m}^{(\alpha,m)}(x),\;\;\; n\ge m\ge 1, \;\;\; \alpha\in {\mathbb R}.
\end{equation}
Therefore, we deduce from  \eqref{alphaint}-\eqref{alphaint2} that for
$ n\ge l+m$ and $ l,m\in {\mathbb N},$
\begin{equation}\label{alphaint3}
P_n^{(-l,-m)}(x)= \Big(\frac{x-1} 2\Big)^l  \Big(\frac{x+1} 2\Big)^m P_{n-l-m}^{(l,m)}(x),
\end{equation}
where we used the fact $d_{n}^{l,-m} d_{n-l}^{m,l}=1.$

For  $\alpha,\beta>-1,$   the (classical) Jacobi polynomials  are orthogonal with respect to the Jacobi weight function:  $\omega^{(\alpha,\beta)}(x) = (1-x)^{\alpha}(1+x)^{\beta},$ namely,
\begin{equation}\label{jcbiorth}
    \int_{-1}^1 {P}_n^{(\alpha,\beta)}(x) {P}_{n'}^{(\alpha,\beta)}(x) \omega ^{(\alpha,\beta)}(x) \, dx= \gamma _n^{(\alpha,\beta)} \delta_{nn'},
\end{equation}
where $\delta_{nn'}$ is the Dirac Delta symbol, and  the normalization constant is given by
\begin{equation}\label{co-gamma}
\gamma _n^{(\alpha,\beta)} =\frac{2^{\alpha+\beta+1}\Gamma(n+\alpha+1)\Gamma(n+\beta+1)}{(2n+\alpha+\beta+1) n!\,\Gamma(n+\alpha+\beta+1)}.
\end{equation}
However, the orthogonality 
does not carry over to the general case. We refer to
 \cite{kuijlaars2005orthogonality} and \cite[Ch. 3]{koekoek2010hypergeometric} for details.

\subsection{Bateman fractional integral formula}
We recall the fractional integral formula of hypergeometric functions  due to Bateman \cite{Bateman1909} (also see \cite[P.  313]{Andrews99}):
for real $c,\rho\ge 0,$ 
  \begin{equation}\label{batemanformsab}
  {}_2F_1(a,b;c+\rho; x)=\frac{\Gamma(c+\rho)}{\Gamma(c)\Gamma(\rho)} x^{1-(c+\rho)}\int_0^x t^{c-1}(x-t)^{\rho-1}
  {}_2F_1(a,b;c; t)\, dt,\quad |x|<1,
  \end{equation}
  where the hypergeometric function $ {}_2F_1$  is defined in  \eqref{hyperboscs}.

The following formulas,   derived from  \eqref{Jacobidefn0} and  \eqref{batemanformsab}  (cf.
\cite[P. 96]{szeg75}),  are  indispensable for the subsequent discussion.
 \begin{lemma}\label{JacobiForm} Let  $\rho\in  {\mathbb R}^+,\; n\in {\mathbb N}_0$ and $x\in \Lambda.$
 \begin{itemize}
 \item[(i)] For  $\alpha>-1$ and $\beta\in {\mathbb R},$
\begin{equation}\label{batemanform0}
(1-x)^{\alpha+\rho} \frac{P_n^{(\alpha+\rho,\beta-\rho)}(x)}{P_n^{(\alpha+\rho,\beta-\rho)}(1)}
=\frac{\Gamma(\alpha+\rho+1)}{\Gamma(\alpha+1)\Gamma(\rho)}\int_{x}^1 \frac{(1-y)^\alpha} {(y-x)^{1-\rho}}
\frac{P_n^{(\alpha,\beta)}(y)}{P_n^{(\alpha,\beta)}(1)}\,  dy.
\end{equation}
\item[(ii)] For  $\alpha\in {\mathbb R}$ and $\beta>-1,$
\begin{equation}\label{batemanform1}
(1+x)^{\beta+\rho} \frac{P_n^{(\alpha-\rho,\beta+\rho)}(x)}{P_n^{(\beta+\rho,\alpha-\rho)}(1)}
=\frac{\Gamma(\beta+\rho+1)}{\Gamma(\beta+1)\Gamma(\rho)}\int_{-1}^x \frac{(1+y)^\beta} {(x-y)^{1-\rho}}
\frac{P_n^{(\alpha,\beta)}(y)}{P_n^{(\beta,\alpha)}(1)}\,  dy.
\end{equation}
\end{itemize}
\end{lemma}
\begin{rem}
 {  The  formulas \eqref{batemanform0}-\eqref{batemanform1}  can be found in several classical   books on orthogonal polynomials,
but it appears that their derivation is not well described.   In fact, taking  $a=-n, b=n+\alpha+\beta+1, c=\alpha+1$ and $t=(1-y)/2$ in \eqref{batemanformsab},
we obtain the formula \eqref{batemanform0}  from   \eqref{Jacobidefn0}. Similarly,  \eqref{batemanform1}  follows from  \eqref{parity} and \eqref{batemanform0}. \qed
}
\end{rem}

Using  the notation in  Definition  \ref{RLFDdefn}  and working out the constants by \eqref{parity},   we can  rewrite the formulas in Lemma \ref{JacobiForm} as follows.
\begin{lmm}\label{JacobiForm2} Let  $\rho\in  {\mathbb R}^+, \; n\in {\mathbb N}_0$ and $x\in \Lambda.$
 \begin{itemize}
 \item For $\alpha>-1$ and $\beta\in {\mathbb R},$
\begin{equation}\label{newbateman}
I_{+}^\rho\big\{(1-x)^\alpha P_n^{(\alpha,\beta)}(x)\big\}=\frac{\Gamma(n+\alpha+1)}{\Gamma(n+\alpha+\rho+1)}
(1-x)^{\alpha+\rho} P_n^{(\alpha+\rho,\beta-\rho)}(x).
\end{equation}
\item For   $ \alpha\in {\mathbb R}$ and
$\beta>-1,$
\begin{equation}\label{newbatemanam}
I_{-}^\rho\big\{(1+x)^\beta P_n^{(\alpha,\beta)}(x)\big\}=\frac{\Gamma(n+\beta+1)}{\Gamma(n+\beta+\rho+1)}
(1+x)^{\beta+\rho} P_n^{(\alpha-\rho,\beta+\rho)}(x).
\end{equation}
\end{itemize}
\end{lmm}

Thanks to  \eqref{rulesa}, we obtain from  Lemma \ref{JacobiForm2}    the following useful ``inverse''  rules.
\begin{lemma}\label{JacobiForm3} Let  $s\in  {\mathbb R}^+, \; n\in {\mathbb N}_0$ and $x\in \Lambda.$
 \begin{itemize}
 \item For $\alpha>-1$ and $\beta\in {\mathbb R},$
\begin{equation}\label{newbateman3}
D_+^s\big\{(1-x)^{\alpha+s} P_n^{(\alpha+s,\beta-s)}(x)\big\}
=\frac{\Gamma(n+\alpha+s+1)} {\Gamma(n+\alpha+1)}(1-x)^\alpha P_n^{(\alpha,\beta)}(x).
\end{equation}
\item For   $ \alpha\in {\mathbb R}$ and $ \beta>-1,$
\begin{equation}\label{newbatemanam3s}
D_-^s\big\{(1+x)^{\beta+s} P_n^{(\alpha-s,\beta+s)}(x)\big\}
=\frac{\Gamma(n+\beta+s+1)} {\Gamma(n+\beta+1)}(1+x)^\beta P_n^{(\alpha,\beta)}(x).
\end{equation}
\end{itemize}
\end{lemma}

Observe that if   $\alpha=0$ in  \eqref{newbateman3}, the fractional derivative operator $D_+^s$ takes
$(1-x)^{s} P_n^{(s,\beta-s)}(x)$ to the  polynomial $P_n^{(0,\beta)}(x).$ Conversely,  if $\alpha+s=k\in {\mathbb N}_0,$  $D_+^s$ takes the polynomial $(1-x)^kP_{n}^{(k,\beta-s)}(x)$ to
$(1-x)^{s-k} P_n^{(s-k,\beta)}(x).$ Such remarkable properties are essential for efficient spectral algorithms to be developed later.
We next show that the above non-polynomial functions  are  intimately  related to  the generalized Jacobi functions  introduced in   \cite{Guo.SW09}.
Moreover,
the Jacobi poly-fractonomials first introduced in  \cite{zayernouri2013fractional} also have direct bearing on these basis functions when  $s\in (0,1).$ 

\section{Generalized Jacobi functions}\label{sect:GJFs}
 In this section, we  modify the definition of  two subclasses of GJFs   in
 \cite{Guo.SW09},   leading  to  the basis functions of interest, which will be  still dubbed  as GJFs.
 We shall  
 demonstrate in Section \ref{sect:appl}   that  spectral algorithms  using   GJF as basis functions   produce  spectral accurate solutions for a class of
 prototypical fractional differential equations.

\subsection{Definition of GJFs}  
\begin{defn}[{\bf Generalized Jacobi functions}]\label{defnGJacbiII} {\em Define
\begin{equation}\label{GJFs+}
 {}^+{\hspace*{-3pt}}J_{n}^{(-\alpha,\beta)}(x):=(1-x)^\alpha P_{n}^{(\alpha,\beta)}(x),\quad {\rm for}\;\;  \alpha>-1, \;\; \beta\in \mathbb{R},
\end{equation}
 and
\begin{equation}\label{GJFs-}
  {}^{-}{\hspace*{-3pt}} J_{n}^{(\alpha,-\beta)}(x):= ({1+x})^\beta P_{n}^{(\alpha,\beta)}(x),\quad   {\rm for}\;\; \alpha\in \mathbb{R},\;\; \beta>-1,
\end{equation}
for  all $x\in \Lambda$ and $n\ge 0.$  
}
\end{defn}

\begin{rem}
Note that  the above definitions  modified the classical Jacobi polynomials in the range of  $-1<\alpha,\beta<1$.  \qed
\end{rem}

Recall the GJFs  introduced  in \cite[(2.7)]{Guo.SW09}:
\begin{equation}\label{jabdefn}
j_n^{(\alpha,\beta)}(x)=\begin{cases}
(1-x)^{-\alpha}(1+x)^{-\beta}P_{\hat n}^{(-\alpha,-\beta)}(x), & (\af,\bt)\in
\aleph_1,\;\; \hat n=n-[-\af]-[-\bt],\\
(1-x)^{-\af}P_{\hat n}^{(-\af,\bt)}(x),& (\af,\bt)\in \aleph_2,\;\;
\hat n=n-[-\af],\\
(1+x)^{-\bt}P_{\hat n}^{(\af,-\bt)}(x),& (\af,\bt)\in
\aleph_3,\;\; \hat n=n-[-\bt],\\
 P_{n}^{(\alpha,\beta)}(x),&(\af,\bt)\in
\aleph_4,
\end{cases}
\end{equation}
where
\begin{equation*}\label{alephset}
\begin{split}
&\aleph_1=\{(\af,\bt): \af,\bt\le  -1 \},\;\;\;
\;\aleph_2=\{(\af,\bt):
\af\le  -1, \;\bt>-1\},\\
&\aleph_3=\{(\af,\bt): \af>-1,\;\bt\le -1 \},\;\;\;\;
\aleph_4=\{(\af,\bt): \af,\bt>-1 \}.
         \end{split}
\end{equation*}
We elaborate below on the connection and difference between the new GJFs and  the GJFs defined in \eqref{jabdefn}.
\begin{itemize}
\item Comparing \eqref{GJFs+}-\eqref{GJFs-}  with \eqref{jabdefn},  we find
\begin{equation}\label{JacobiGJFs}
\begin{split}
&{}^+{\hspace*{-3pt}}J_{n}^{(-\alpha,\beta)}(x)= j_{n+[\alpha]}^{(-\alpha,\beta)}(x), \quad {\rm if}\;\; \alpha\ge 1, \;\; \beta>-1;\\
& {}^-{\hspace*{-3pt}}J_{n}^{(\alpha,-\beta)}(x)= j_{n+[\beta]}^{(\alpha,-\beta)}(x), \quad {\rm if}\;\; \alpha>-1, \;\; \beta\ge 1.
\end{split}
\end{equation}
\item By \eqref{alphaint}-\eqref{alphaint2}, we  find from  \eqref{GJFs+}-\eqref{GJFs-} that  for any $\alpha>-1, k\in {\mathbb N}_0$ and $n\ge k,$
 \begin{equation}\label{zerokpadd}
 \begin{split}
& {}^+ {\hspace*{-3pt}}J_{n}^{(-\alpha, -k)}(x)=2^{-k}d_n^{k,\alpha} (1-x)^\alpha(1+x)^kP_{n-k}^{(\alpha,k)}(x); \\
& {}^- {\hspace*{-3pt}}J_{n}^{(-k, -\alpha)}(x)=(-1)^k2^{-k}d_n^{k,\alpha} (1-x)^k(1+x)^\alpha P_{n-k}^{(k,\alpha)}(x),
\end{split}
\end{equation}
which, compared with  \eqref{jabdefn}, implies that  for $\alpha\ge 1$ and $n\ge k\ge 1,$
 \begin{equation}\label{zerokpadd2A}
 \begin{split}
& {}^+ {\hspace*{-3pt}}J_{n}^{(-\alpha, -k)}(x)= 2^{-k}d_n^{k,\alpha} j_{n+[\alpha]}^{(-\alpha,-k)}(x);\quad
 {}^- {\hspace*{-3pt}}J_{n}^{(-\alpha, -k)}(x)=(-1)^k2^{-k}d_n^{k,\alpha} j_{n+[\alpha]}^{(-k,-\alpha)}(x).
\end{split}
\end{equation}
Here, the constant $d_n^{k,\alpha}$ is defined in \eqref{alphaint}.
 \end{itemize}
We see that we  modified  the  definition of GJFs in \cite{Guo.SW09}  for the parameters in the ranges  other than those specified in
\eqref{JacobiGJFs} and \eqref{zerokpadd2A}. Indeed, this opens up new applicability of the GJFs in solving fractional differential equations, see Section \ref{sect:appl}.
\subsection{Properties of GJFs}   One verifies readily from \eqref{parity} and Definition \ref{defnGJacbiII} that  for  $\alpha>-1$ and  $\beta\in \mathbb{R}$,
\begin{equation}\label{parity2}
{}^+{\hspace*{-3pt}}J_{n}^{(-\alpha,\beta)}(-x)=(-1)^{n} \, {}^-{\hspace*{-3pt}}J_{n}^{(\beta,-\alpha)}(x),
\end{equation}
and for $-1<\alpha<1,$  there holds  the reflection property:
\begin{equation}\label{alphaA}
{}^+{\hspace*{-3pt}}J_{n}^{(-\alpha,-\alpha)}(x)=(1-x^2)^\alpha\;  {}^{-}{\hspace*{-3pt}} J_{n}^{(\alpha,\alpha)}(x).
\end{equation}

If  $-(n+\alpha+\beta)\not \in {\mathbb N},$ we can use  \eqref{ja3term} to evaluate  ${}^+{\hspace*{-3pt}}J_{n}^{(-\alpha,\beta)}$ recursively:
\begin{equation}\label{ja3termNew}
\begin{split}
&{}^+{\hspace*{-3pt}}J_{n+1}^{(-\alpha,\beta)}(x)=\big(a_{n}^{\alpha,\beta}x-b_{n}^{\alpha,\beta}\big){}^+{\hspace*{-3pt}}J_{n}^{(-\alpha,\beta)}(x)
-c_{n}^{\alpha,\beta}{}^+{\hspace*{-3pt}}J_{n-1}^{(-\alpha,\beta)}(x),\;\; n\ge 1,\\
&{}^+{\hspace*{-3pt}}J_{0}^{(-\alpha,\beta)}(x)=(1-x)^\alpha,\;\;\; {}^+{\hspace*{-3pt}}J_{1}^{(-\alpha,\beta)}(x)=
\big((\alpha+\beta+2)x+\alpha-\beta\big)(1-x)^{\alpha}/2,
\end{split}
\end{equation}
where $a_{n}^{\alpha,\beta}, b_{n}^{\alpha,\beta}, c_{n}^{\alpha,\beta}$ are defined in \eqref{jacoefabc}.  Accordingly,  we can compute  ${}^-{\hspace*{-3pt}}J_{n}^{(\alpha,-\beta)}(x)$   by   \eqref{parity2}.

We now study   the orthogonality  of  GJFs.  It follows straightforwardly from \eqref{jcbiorth} and Definition
\ref{defnGJacbiII} that for $\alpha,\beta>-1,$
\begin{equation}\label{basicorth}
\begin{split}
\int_{-1}^1 & {}^+{\hspace*{-3pt}} J_{n}^{(-\alpha,\beta)}(x) {}^+{\hspace*{-3pt}}J_{n'}^{(-\alpha,\beta)}(x)\,  \omega^{(-\alpha,\beta)}(x)\,dx\\
& =\int_{-1}^1 {}^-{\hspace*{-3pt}}J_{n}^{(\alpha,-\beta)}(x) {}^-{\hspace*{-3pt}}J_{n'}^{(\alpha,-\beta)}(x) \, \omega^{(\alpha,-\beta)}(x)\,dx=\gamma_n^{(\alpha,\beta)} \delta_{nn'},
\end{split}
 \end{equation}
 where $\gamma_n^{(\alpha,\beta)}$ is defined in \eqref{co-gamma}.  Similarly,  by \eqref{jcbiorth} and \eqref{zerokpadd}, we have that   for $\alpha>-1$ and $k\in {\mathbb N},$
\begin{equation}\label{orthogonalityII}
\begin{split}
\int_{-1}^1 &  {}^+{\hspace*{-3pt}}J_{n}^{(-\alpha,-k)} (x)\,
 {}^+{\hspace*{-3pt}}J_{n'}^{(-\alpha,-k)} (x) \,\omega^{(-\alpha,-k)}(x) \,dx \\
 &=
 \int_{-1}^1  {}^-{\hspace*{-3pt}}J_{n}^{(-k,-\alpha)} (x)\,
 {}^-{\hspace*{-3pt}}J_{n'}^{(-k,-\alpha)} (x)\, \omega^{(-k,-\alpha)} (x) \,dx = \gamma^{(\alpha,-k)}_{n}\delta_{nn'},\quad n, n'\geq k,
\end{split}
\end{equation}
where we used the fact
 $$\gamma^{(\alpha,-k)}_{n} =2^{-2k}(d_n^{k,\alpha})^2\gamma_{n-k}^{(\alpha,k)}.$$

Next, we discuss  the fractional calculus properties  of  GJFs.  The following fractional derivative formulas  can be derived straightforwardly from  Lemma \ref{JacobiForm3} and Definition \ref{defnGJacbiII}.
\begin{thm}\label{JacobiFormGJFs} Let  $s\in  {\mathbb R}^+, \; n\in {\mathbb N}_0$ and $x\in \Lambda.$
 \begin{itemize}
 \item For $\alpha>s-1$ and $\beta\in {\mathbb R},$
\begin{equation}\label{derivative+}
D_+^s\big\{{}^+{\hspace*{-3pt}}J_{n}^{(-\alpha,\beta)}(x)\big\}
=\frac{\Gamma(n+\alpha+1)} {\Gamma(n+\alpha-s+1)}{}^+{\hspace*{-3pt}}J_{n}^{(-\alpha+s,\beta+s)}(x).
\end{equation}
\item For   $ \alpha\in {\mathbb R}$ and $ \beta>s-1,$
\begin{equation}\label{derivative-}
D_-^s\big\{{}^{-}{\hspace*{-3pt}} J_{n}^{(\alpha,-\beta)}(x)\big\}
=\frac{\Gamma(n+\beta+1)} {\Gamma(n+\beta-s+1)}{}^{-}{\hspace*{-3pt}} J_{n}^{(\alpha+s,-\beta+s)}(x).
\end{equation}
\end{itemize}
\end{thm}
\vskip 2pt
\noindent Some remarks on Theorem  \ref{JacobiFormGJFs}  are in order.
\begin{itemize}
\item If $\alpha-s>-1$ and $\beta+s>-1$ with $s\in  {\mathbb R}^+,$ then by \eqref{basicorth} and \eqref{derivative+},     $\big\{D_+^s {}^+{\hspace*{-3pt}}J_{n}^{(-\alpha,\beta)}\big\}$ are mutually orthogonal with respect to the weight function $\omega^{(-\alpha+s, \beta+s)}(x).$ Similarly,  $\big\{D_-^s {}^-{\hspace*{-3pt}}J_{n}^{(\alpha,-\beta)}\big\}$ are mutually orthogonal  with respect to $\omega^{(\alpha+s, -\beta+s)}(x),$ when  $\alpha+s>-1$ and $\beta-s>-1.$
\item A very important special case of \eqref{derivative+}  is that  for  $\alpha>0$ and $\beta\in {\mathbb R},$
\begin{equation}\label{derivative+A}
D_+^\alpha\big\{{}^+{\hspace*{-3pt}}J_{n}^{(-\alpha,\beta)}(x)\big\}
=\frac{\Gamma(n+\alpha+1)} {n!}{}^+{\hspace*{-3pt}}J_{n}^{(0,\alpha+\beta)}(x)=\frac{\Gamma(n+\alpha+1)} {n!}P_n^{(0,\alpha+\beta)}(x).
\end{equation}
Similarly, by \eqref{derivative-}, we have that  for $\alpha\in {\mathbb R}$ and real $\beta>0,$
\begin{equation}\label{derivative+B}
D_-^\beta\big\{{}^-{\hspace*{-3pt}}J_{n}^{(\alpha,-\beta)}(x)\big\}
=\frac{\Gamma(n+\beta+1)} {n!}P_n^{(\alpha+\beta,0)}(x).
\end{equation}
These two formulas  indicate that performing a suitable order of fractional derivatives on   GJFs leads to polynomials.
\end{itemize}

The analysis of  the approximability of GJFs essentially relies on the orthogonality of fractional derivatives of GJFs.   To study this,
we first recall  the derivative formula of  the classical Jacobi polynomials (see,  e.g., \cite[P. 72]{ShenTangWang2011}):  for $\alpha,\beta>-1$ and $n\ge l, $
\begin{equation}\label{derimulti}
D^l P_n^{(\alpha,\beta)}(x)= \kappa_{n,l}^{(\alpha,\beta)} P_{n-l}^{(\alpha+l,\beta+l)}(x),\;\;  {\rm where}\;\;   \kappa_{n,l}^{(\alpha,\beta)}:=\frac{\Gamma(n+\alpha+\beta+l+1)}{2^l\Gamma(n+\alpha+\beta+1)}.
\end{equation}
Noting that $D^{s+l}_\pm=(\mp1)^lD^l D_\pm^s,$ we derive from \eqref{jcbiorth} and \eqref{derivative+A}-\eqref{derimulti}  the following orthogonality.
\begin{itemize}
\item For $\alpha>0$ and $\alpha+\beta>-1, $
\begin{equation}\label{D+orthl}
\begin{split}
\int_{-1}^1 & D_+^{\alpha+l}\, {}^+{\hspace*{-3pt}} J_{n}^{(-\alpha,\beta)}(x)\,  D_+^{\alpha+l}\, {}^+{\hspace*{-3pt}}J_{n'}^{(-\alpha,\beta)}(x)\,  \omega^{(l,\alpha+\beta+l)}(x)\,dx=h_{n,l}^{(\alpha,\beta)} \delta_{nn'},\;\;\; n,n'\ge l\ge0,
\end{split}
 \end{equation}
 where
 \begin{equation}\label{constanth}
\begin{split}
h_{n, l}^{(\alpha,\beta)}&:=\frac{\Gamma^2(n+\alpha+1)}{(n!)^2} \big(\kappa_{n,l}^{(0,\alpha+\beta)}\big)^2\gamma_{n-l}^{(l,\alpha+\beta+l)}\\
&=
\frac{2^{\alpha+\beta+1}\Gamma^2(n+\alpha+1)\,\Gamma(n+\alpha+\beta+l+1)}{(2n+\alpha+\beta+1)\,n!\,(n-l)!\, \Gamma(n+\alpha+\beta+1)}.
\end{split}
\end{equation}
\item For $\alpha+\beta>-1$ and $\beta>0, $
\begin{equation}\label{D-orthl}
\begin{split}
\int_{-1}^1 & D_-^{\beta+l}\, {}^-{\hspace*{-3pt}} J_{n}^{(\alpha,-\beta)}(x)\,  D_-^{\beta+l}\, {}^-{\hspace*{-3pt}}J_{n'}^{(\alpha,-\beta)}(x)\,  \omega^{(\alpha+\beta+l,l)}(x)\,dx=h_{n,l}^{(\beta,\alpha)} \delta_{nn'},\;\;\; n,n'\ge l\ge0.
\end{split}
 \end{equation}
\end{itemize}

\vskip 4pt
Another attractive  property of   GJFs is that they are eigenfunctions of  fractional  Sturm-Liouville-type equations.  To show this,
we define the fractional Sturm-Liouville-type  operators:
\begin{equation}\label{SLprobJGF+AB}
{}^+{\hspace*{-3pt}}{\mathcal L}^{2s}_{\alpha,\beta}u:= \omega^{(\alpha,-\beta)} D_-^s \big\{ \omega^{(-\alpha+s,\beta+s)} D_+^s  u\big\};\;\;
{}^-{\hspace*{-3pt}}{\mathcal L}^{2s}_{\alpha,\beta}u:= \omega^{(-\alpha,\beta)} D_+^s \big\{ \omega^{(\alpha+s,-\beta+s)} D_-^s  u\big\}.
 \end{equation}
\begin{thm}\label{thm:SLprob} Let $s\in {\mathbb R}^+, \; n\in {\mathbb N}_0 $ and $x\in \Lambda$. 
\begin{itemize}
\item For $\alpha>s-1$ and $\beta>-1,$  
\begin{equation}\label{SLprobJGF+}
{}^+{\hspace*{-3pt}}{\mathcal L}^{2s}_{\alpha,\beta}{}^+{\hspace*{-3pt}}J_{n}^{(-\alpha,\beta)}(x)
=\lambda_{n,s}^{(\alpha,\beta)}\, {}^+{\hspace*{-3pt}}J_{n}^{(-\alpha,\beta)}(x),
 \end{equation}
 where  
 \begin{equation}\label{abweights}
 \lambda_{n,s}^{(\alpha,\beta)}:=\frac{\Gamma(n+\alpha+1)} {\Gamma(n+\alpha-s+1)} \frac{\Gamma(n+\beta+s+1)} {\Gamma(n+\beta+1)}.
 \end{equation}
 \item For $\alpha>-1$ and $\beta>s-1,$
\begin{equation}\label{SLprobJGF-}
{}^-{\hspace*{-3pt}}{\mathcal L}^{2s}_{\alpha,\beta}{}^-{\hspace*{-3pt}}J_{n}^{(\alpha,-\beta)}(x)
=\lambda_{n,s}^{(\beta,\alpha)}\, {}^-{\hspace*{-3pt}}J_{n}^{(\alpha,-\beta)}(x).
 \end{equation}
\end{itemize}
 \end{thm}
\begin{proof}  By  Definition \ref{defnGJacbiII} and  \eqref{derivative+},  we have that for $\alpha>s-1,$
\begin{equation}\label{tempA}
\begin{split}
(1-x)^{-\alpha+s}(1+x)^{\beta+s} & D_+^s \big\{{}^+{\hspace*{-3pt}}J_{n}^{(-\alpha,\beta)}(x)\big\}
=\frac{\Gamma(n+\alpha+1)} {\Gamma(n+\alpha-s+1)}  \, (1+x)^{\beta+s}P_n^{(\alpha-s,\beta+s)}(x).
\end{split}
\end{equation}
Applying $D_-^s$ on both sides of the above identity and tracking the constants, we derive from \eqref{newbatemanam3s}    that for $\beta>-1,$
 \begin{equation*}
 \begin{split}
D_-^s \big\{ \omega^{(-\alpha+s,\beta+s)}(x) &D_+^s  \big\{{}^+{\hspace*{-3pt}}J_{n}^{(-\alpha,\beta)}(x)\big\}\big\}=
 \lambda_{n,s}^{(\alpha,\beta)}\, (1+x)^{\beta} P_n^{(\alpha,\beta)}(x)\\
 &=\lambda_{n,s}^{(\alpha,\beta)}\, \omega^{(-\alpha,\beta)} (x)\,{}^+{\hspace*{-3pt}}J_{n}^{(-\alpha,\beta)}(x).
\end{split}
 \end{equation*}
 This yields   \eqref{SLprobJGF+}.

The property  \eqref{SLprobJGF-}  can be proved  in a very  similar fashion.
\end{proof}
\begin{rem}\label{SLspecial}{
The above results can be viewed as an extension of the standard Sturm-Liouville problems  of   GJFs  to the fractional derivative case.  In \cite{Guo.SW09},   we showed that GJFs defined  therein are the eigenfunctions of  the standard Sturm-Liouville problems. \qed}
\end{rem}
\begin{rem}
 We derive immediately from \eqref{abweights} and the Stirling's formula (see \eqref{stirlingfor} below)  that for fixed $s, \alpha,\beta,$
 $$\lambda_{n,s}^{(\alpha,\beta)}=O(n^{2s}), \quad {\rm for}\;\; n\gg 1.$$
 When $s\to 1,$  this  recovers  the $O(n^2)$ growth of   eigenvalues of the standard Sturm-Liouville problem. \qed
\end{rem}

Note that the fractional Sturm-Liouville operators defined in \eqref{SLprobJGF+AB} are not self-adjoint in general.  However, the singular fractional Sturm-Liouville problems are self-adjoint, when $s\in (0,1).$

\begin{cor}\label{addcor} Let $s\in (0,1), \; n\in {\mathbb N}_0 $ and $x\in \Lambda.$
\begin{itemize}
\item For $0<\alpha<s$ and $\beta>-s,$  we have that in \eqref{SLprobJGF+},
\begin{equation}\label{adjointA}
{}^+{\hspace*{-3pt}}{\mathcal L}^{2s}_{\alpha,\beta}{}^+{\hspace*{-3pt}}J_{n}^{(-\alpha,\beta)} = \omega^{(\alpha,-\beta)} D_-^s \big\{ \omega^{(-\alpha+s,\beta+s)} {}^C\hspace*{-2pt}D_+^s  {}^+{\hspace*{-3pt}}J_{n}^{(-\alpha,\beta)}\big\},\;\;
 \end{equation}
and
 \begin{equation}\label{adjointpropA}
 \begin{split}
 & \big({}^+\hspace*{-2pt}{\mathcal L}_{\alpha,\beta}^{2s}\, {}^+ {\hspace*{-3pt}}J_{n}^{(-\alpha,\beta)}, {}^+{\hspace*{-3pt}}J_{m}^{(-\alpha,\beta)}\big)_{\omega^{(-\alpha,\beta)}}=
\big({^C}\hspace*{-2pt}D_+^s\,{}^+{\hspace*{-3pt}}J_{n}^{(-\alpha,\beta)}, {^C}\hspace*{-2pt}D_+^s\,{}^+{\hspace*{-3pt}}J_{m}^{(-\alpha,\beta)}  \big)_{\omega^{(-\alpha-s,\beta+s)}}\\
& =\big({}^+{\hspace*{-3pt}}J_{n}^{(-\alpha,\beta)}, {}^+\hspace*{-2pt}{\mathcal L}_{\alpha,\beta}^{2s}\, {}^+{\hspace*{-3pt}}J_{m}^{(-\alpha,\beta)}\big)_{\omega^{(-\alpha,\beta)}}= \lambda_{n,s}^{(\alpha,\beta)}\gamma_n^{(-\alpha,\beta)}
\delta_{nm}.
\end{split}
 \end{equation}
\item Similarly,  for $\alpha>-s$ and $0<\beta<s,$ we have that in  \eqref{SLprobJGF-},
\begin{equation}\label{adjointB}
{}^-{\hspace*{-3pt}}{\mathcal L}^{2s}_{\alpha,\beta}{}^- {\hspace*{-3pt}}J_{n}^{(\alpha,-\beta)} = \omega^{(-\alpha,\beta)} D_+^s \big\{ \omega^{(\alpha+s,-\beta+s)} {}^C\hspace*{-2pt}D_-^s  {}^- {\hspace*{-3pt}}J_{n}^{(\alpha,-\beta)}\big\},
\end{equation}
and
 \begin{equation}\label{adjointpropB}
 \begin{split}
 & \big({}^-\hspace*{-2pt}{\mathcal L}_{\alpha,\beta}^{2s}\, {}^- {\hspace*{-3pt}}J_{n}^{(\alpha,-\beta)}, {}^-{\hspace*{-3pt}}J_{m}^{(\alpha,-\beta)}\big)_{\omega^{(\alpha,-\beta)}}=
\big({^C}\hspace*{-2pt}D_-^s\,{}^-{\hspace*{-3pt}}J_{n}^{(\alpha,-\beta)}, {^C}\hspace*{-2pt}D_-^s\,{}^-{\hspace*{-3pt}}J_{m}^{(\alpha,-\beta)}  \big)_{\omega^{(\alpha+s,-\beta-s)}}\\
& =\big({}^-{\hspace*{-3pt}}J_{n}^{(\alpha,-\beta)}, {}^-\hspace*{-2pt}{\mathcal L}_{\alpha,\beta}^{2s}\, {}^-{\hspace*{-3pt}}J_{m}^{(\alpha,-\beta)}\big)_{\omega^{(\alpha,-\beta)}}= \lambda_{n,s}^{(\beta,\alpha)}\gamma_n^{(\alpha,-\beta)}
\delta_{nm}.
\end{split}
 \end{equation}
 \end{itemize}
\end{cor}
\begin{proof}
We just prove the results for ${}^+{\hspace*{-3pt}}J_{n}^{(-\alpha,\beta)}(x).$
For $\alpha>0$ and $s\in (0,1),$ since ${}^+{\hspace*{-3pt}}J_{n}^{(-\alpha,\beta)}(1)=0,$   we find from \eqref{RL=Ca} that
$D_+^s$ can be replaced by ${^C}\hspace*{-2pt}D_+^s.$  Accordingly,  \eqref{adjointA} follows from
\eqref{SLprobJGF+} immediately.

 We now show the fractional integration by parts can get through.  By \eqref{newbatemanam} and \eqref{tempA},
 \begin{equation}\label{salphas}
 I_-^{1-s}\big\{\omega^{(-\alpha+s,\beta+s)}\, {^C}\hspace*{-2pt}D_+^s\,{}^+{\hspace*{-3pt}}J_{n}^{(-\alpha,\beta)}\big\}=\tilde d_{n,s}^{\alpha,\beta}  \, (1+x)^{\beta+1}P_n^{(\alpha-1,\beta+1)}(x),
 \end{equation}
 where the constant $\tilde d_{n,s}^{\alpha,\beta}$ can be worked out.  Clearly, it vanishes at $x=-1.$ On the other
 hand,  ${}^+{\hspace*{-3pt}}J_{m}^{(-\alpha,\beta)}(1)=0.$  Therefore, we can perform
 the rule \eqref{left-byparts} to obtain the second identity in \eqref{adjointpropA}.  The orthogonality follows from  \eqref{basicorth} and \eqref{adjointA}.

 The results for ${}^-{\hspace*{-3pt}}J_{n}^{(\alpha,-\beta)}(x)$ can be derived similarly.
\end{proof}

\subsection{Relation with Jacobi poly-fractonomials} In a very recent paper,
 Zayernouri and Karniadakis
\cite{zayernouri2013fractional} introduced a family of Jacobi poly-fractonomials (JPFs) from  the eigenfunctions of a singular  factional Sturm-Liouville problem. We first recall their definition.
\begin{defn}[\bf Jacobi poly-fractonomials \cite{zayernouri2013fractional}]\label{JPFss} {\em For $\mu\in (0,1), $  the Jacobi poly-fractonomials of order $\mu$ are defined as follows.
\begin{itemize}
\item For $-1<\alpha<2-\mu$ and $-1<\beta<\mu-1,$
\begin{equation}\label{JPFs-}
\begin{split}
& {^{(1)}\hspace*{-1pt}{\mathcal P}_n^{(\alpha,\beta,\mu)}}(x)=(1+x)^{\mu-(\beta+1)}{P}_{n-1}^{(\alpha+1-\mu,\mu-(\beta+1))}(x),\quad n\ge 1.
\end{split}
\end{equation}
\item For $-1<\alpha<\mu-1$ and $-1<\beta<2-\mu,$
\begin{equation}\label{JPFs+}
  {^{(2)}\hspace*{-1pt}{\mathcal P}_n^{(\alpha,\beta,\mu)}}(x) =(1-x)^{\mu-(\alpha+1)}{P}_{n-1}^{(\mu-(\alpha+1),\beta+1-\mu)}(x), \quad n\ge 1.
\end{equation}

\end{itemize}
}
\end{defn}

As shown in \cite[Thm. 4.2]{zayernouri2013fractional},  the left JPFs are eigenfunctions of  the singular fractional Sturm-Liouville equation:
\begin{equation}\label{singularSL}
{D}^\mu_{+}\big\{\omega^{(\alpha+1,\beta+1)}(x) {^C}\hspace*{-2pt}D_-^\mu \{ {^{(1)}\hspace*{-1pt}{\mathcal P}_n^{(\alpha,\beta,\mu)}}(x) \} \big\}={^{(1)}}\hspace*{-2pt}\lambda_n^{(\alpha,\beta,\mu)}\,
\omega^{(\alpha+1-\mu,\beta+1-\mu)}(x) \,{^{(1)}\hspace*{-1pt}{\mathcal P}_n^{(\alpha,\beta,\mu)}}(x),
\end{equation}
where
$$ {^{(1)}}\hspace*{-2pt}\lambda_n^{(\alpha,\beta,\mu)}=\frac{\Gamma(n+\alpha+1)\Gamma(n+\mu-\beta-1)}{\Gamma(n-\beta-1)\Gamma(n-\mu+\alpha+1)},\quad n\ge 1.$$
The right JPFs  satisfy a similar  equation. 

It follows from \eqref{GJFs+}-\eqref{GJFs-} and  \eqref{JPFs-}-\eqref{JPFs+}  the  relation:
  \begin{equation}\label{relaKarnia}
  {^{(1)}\hspace*{-1pt}{\mathcal P}_n^{(\alpha,\beta,\mu)}}(x)={}^-{\hspace*{-3pt}}J_{n-1}^{(\alpha+1-\mu,\beta+1-\mu)}(x),\quad
  {^{(2)}\hspace*{-1pt}{\mathcal P}_n^{(\alpha,\beta,\mu)}}(x)={}^+{\hspace*{-3pt}}J_{n-1}^{(\alpha+1-\mu,\beta+1-\mu)}(x).
  \end{equation}
  Observe that with the parameters  $\{\mu,\alpha+1-\mu,\mu-(\beta+1)\}$ in place of
  $\{s,\alpha,\beta\}$ in \eqref{adjointB}, we obtain \eqref{singularSL} exactly.
  However,  the range of the parameters is $\alpha>-1$ and $-1<\beta<1-\mu,$ so the condition on $\alpha$ is relaxed as opposite to that for \eqref{JPFs-}. Indeed, the difference between  the range of  $\alpha$  is not surprising,  as  the GJFs here  and JPFs in
   \cite{zayernouri2013fractional}  are defined by different means.

\section{Approximation by  GJFs}\label{sect:appr}
\setcounter{equation}{0}
\setcounter{lmm}{0}
\setcounter{thm}{0}


The main concern of this section is to show that  approximation by  GJF series  leads to typical spectral convergence for functions in appropriate
weighted Sobolev spaces involving fractional derivatives.    Such approximation results play a crucial role in the  analysis of spectral methods for fractional
differential equations, see Section \ref{sect:appl}.

For simplicity  of presentation, we only provide the detailed analysis for  $\big\{{}^+ {\hspace*{-3pt}}J_{n}^{(-\alpha,\beta)}\big\},$ as the results  can be extended to  $\big\{{}^- {\hspace*{-3pt}}J_{n}^{(\alpha,-\beta)}\big\}$ straightforwardly, thanks to \eqref{parity2}.
In the first place, we highlight some special GJFs  of  particular  interest.
\begin{itemize}
 \item For $\alpha>0$ and $\beta\in {\mathbb R}$ (such that  $-(n+\alpha+\beta)\not \in {\mathbb N}$ and  $-\beta\not\in {\mathbb N}$),  we have
\begin{equation}\label{harmocondA}
D^l\,{}^+{\hspace*{-3pt}}J_{n}^{(-\alpha,\beta)}(1)=0, \quad 
{\rm for}\;\;  l=0,1,\cdots, [\alpha]-1,
\end{equation}
which naturally allows us to impose the one-sided boundary conditions:  $u^{(l)}(1)=0$  for $l=0,1,\cdots, [\alpha]-1,$ and more importantly,  it matches the singularity of the solution for prototypical  fractional initial value problems, thanks to the fractional factor  $(1-x)^{\alpha}.$  Moreover, we can choose the parameter $\beta$ (e.g., $\beta=-\alpha$)  so that under the GJF basis, the linear systems of the fractional equations can be sparse and well-conditioned.
\item For $\alpha>0$ and   $\beta=-[\alpha]$,   we find from \eqref{zerokpadd} that  for $n\ge [\alpha],$ 
\begin{equation}\label{harmocondB}
D^l\,{}^+{\hspace*{-3pt}}J_{n}^{(-\alpha,-[\alpha])}(\pm 1)=0,\quad {\rm for}\;\;\;  l=0,1,\cdots, [\alpha]-1,
\end{equation}
which allows us to deal with two-sided boundary conditions:  $u^{(l)}(\pm 1)=0,$  and to match the singularity  of the solution to some prototypical fractional boundary value problems.
\end{itemize}

We introduce some notation to be used later. Let ${\mathcal P}_N$ be the set of all algebraic (real-valued) polynomials  of degree at most $N.$ Let $\varpi(x)>0$ for all $x\in \Lambda,$ be a generic weight function.  The weighted space  $L^2_\varpi(\Lambda)$ is defined as in  Admas \cite{Adam75} with the inner product  and norm
$$(u,v)_\varpi=\int_\Lambda u(x) v(x)\varpi(x) dx,\quad \|u\|_\varpi=(u,u)_\varpi^{1/2}. $$
If $\varpi\equiv 1,$ we omit the weight in the notation.  In what follows, the Sobolev space $H^1(\Lambda)$ is also defined as usual.
\subsection{Approximation results for GJFs $\big\{{}^+ {\hspace*{-3pt}}J_{n}^{(-\alpha,\beta)}\big\}$}
In view of the applications that we have in mind,  we restrict the parameters to the set
\begin{equation}\label{salphacond}
{}^+\hspace*{-1pt}\Upsilon^{\alpha,\beta}:=\big\{(\alpha,\beta)\,:\, \alpha> 0,\;\;  \alpha+\beta>-1\big\},
\end{equation}
 which we  further split  into three disjoint subsets:
\begin{equation}\label{3subcase}
\begin{split}
&  {}^+\hspace*{-1pt} \Upsilon_1^{\alpha,\beta}:= \big\{(\alpha,\beta)\,:\, \alpha> 0,\; \beta>-1\big\};\\
& {}^+\hspace*{-1pt} \Upsilon_2^{\alpha,\beta}:= \big\{(\alpha,\beta)\,:\, \alpha> 0,\; -\alpha-1<\beta=-k\le -1,\;  k\in {\mathbb N}\big\};\\
& {}^+\hspace*{-1pt} \Upsilon_3^{\alpha,\beta}:= \big\{(\alpha,\beta)\,:\, \alpha> 0,\; -\alpha-1<\beta<-1,\;  -\beta\not\in {\mathbb N}\big\}.
\end{split}
\end{equation}

\subsubsection{Case I: $(\alpha,\beta)\in {}^+\hspace*{-1pt}\Upsilon_1^{\alpha,\beta}\cup {}^+\hspace*{-1pt}\Upsilon_2^{\alpha,\beta}$.} Let us first consider
$(\alpha,\beta)\in {}^+\hspace*{-1pt}\Upsilon_1^{\alpha,\beta}$. In this case, we define the finite-dimensional  fractional-polynomial space:
\begin{equation}\label{frac-polysps}
{}^+\hspace*{-2pt} {\mathcal F}_N^{(-\alpha,\beta)}(\Lambda)=\big\{\phi=(1-x)^\alpha \psi\,:\, \psi\in {\mathcal P}_N\big\}={\rm span}\big\{{}^+{\hspace*{-3pt}}J_{n}^{(-\alpha,\beta)}\,:\, 0\le n\le N \big\}.
\end{equation}
By the orthogonality  \eqref{basicorth},  we can expand any $u\in L^2_{\omega^{(-\alpha,\beta)}}(\Lambda)$ as
\begin{equation}\label{uexpansionA}
u(x)=\sum_{n=0}^\infty \hat u_n^{(\alpha,\beta)} \,  {}^+{\hspace*{-3pt}}J_{n}^{(-\alpha,\beta)}(x),\;\;  {\rm where}\;\;
\hat u_n^{(\alpha,\beta)}=\frac 1 {\gamma_n^{(\alpha,\beta)}} \int_{-1}^1 u \; {}^+{\hspace*{-3pt}}J_{n}^{(-\alpha,\beta)}\omega^{(-\alpha,\beta)}\, dx,
\end{equation}
and there holds the Parseval identity:
\begin{equation}\label{uexpansionB}
\|u\|_{{\omega^{(-\alpha,\beta)}}}^2=\sum_{n=0}^\infty  \gamma_n^{(\alpha,\beta)}\big|\hat u_n^{(\alpha,\beta)}\big|^2.
\end{equation}
Consider the $L^2_{\omega^{(-\alpha,\beta)}}$-orthogonal projection upon ${}^+\hspace*{-2pt} {\mathcal F}_N^{(-\alpha,\beta)}(\Lambda),$ defined by
\begin{equation}\label{projectorII3}
\big({}^+\hspace*{-2pt}\pi_N^{(-\alpha,\beta)}u-u,  v_N\big)_{\omega^{(-\alpha,\beta)}}=0,\quad \forall\,
v_N\in  {}^+\hspace*{-2pt} {\mathcal F}_N^{(-\alpha,\beta)}(\Lambda).
\end{equation}
By definition, we have
\begin{equation}\label{projectorIII}
({}^+\hspace*{-2pt}\pi_N^{(-\alpha,\beta)}u\big)(x)=\sum_{n=0}^N \hat u_n^{(\alpha,\beta)} \,  {}^+{\hspace*{-3pt}}J_{n}^{(-\alpha,\beta)}(x).
\end{equation}

We now consider  $(\alpha,\beta)\in {}^+\Upsilon_2^{\alpha,\beta}$.   In this case, we modify  \eqref{frac-polysps} as
\begin{equation}\label{frac-polyspsCase2}
{}^+\hspace*{-2pt} {\mathcal F}_N^{(-\alpha,-k)}(\Lambda)=
\big\{\phi=(1-x)^\alpha \psi\,:\, \psi\in {\mathcal P}_N\; \text{such that} \; \psi^{(l)}(-1)=0,\; 0\le l\le k-1\big\},
\end{equation}
which incorporates the homogeneous boundary conditions at $x=-1.$
Thanks to \eqref{zerokpadd}, we have
\begin{equation}\label{spanrela}
{}^+\hspace*{-2pt} {\mathcal F}_N^{(-\alpha,-k)}(\Lambda)={\rm span}\big\{{}^+{\hspace*{-3pt}}
J_{n}^{(-\alpha,-k)}(x)\,:\, k\le n\le N \big\}.
\end{equation}
In  view of the orthogonality \eqref{orthogonalityII},  we have the expansion like \eqref{uexpansionA}, that is,   for any $u\in L^2_{\omega^{(-\alpha,-k)}}(\Lambda),$
\begin{equation}\label{uexpansionAB}
u(x)=\sum_{n=k}^\infty \hat u_n^{(\alpha,-k)} \,  {}^+{\hspace*{-3pt}}J_{n}^{(-\alpha,-k)}(x),\;\;  {\rm where}\;\;
\hat u_n^{(\alpha,-k)}=\frac 1 {\gamma_n^{(\alpha,-k)}} \int_{-1}^1 u\;  {}^+{\hspace*{-3pt}}J_{n}^{(-\alpha,-k)}\omega^{(-\alpha,-k)}\, dx,
\end{equation}
so the identity \eqref{uexpansionB} also holds for this expansion.
The partial sum
\begin{equation}\label{pinuexp}
{}^+\hspace*{-2pt}\pi_N^{(-\alpha,-k)}u(x)=\sum_{ n=k}^{N} \hat u_{ n}^{(\alpha,-k)}\,
 {}^+{\hspace*{-3pt}}J_{n}^{(-\alpha,-k)} (x),
\end{equation}
is the  $L^2_{\omega^{(-\alpha,-k)}}$-orthogonal projection upon ${}^+\hspace*{-2pt} {\mathcal F}_N^{(-\alpha,-k)}(\Lambda),$ namely,
\begin{equation}\label{projectorII}
\big({}^+\hspace*{-2pt}\pi_N^{(-\alpha,-k)}u-u,  v_N\big)_{\omega^{(-\alpha,-k)}}=0,\quad \forall\,
v_N\in  {}^+\hspace*{-2pt} {\mathcal F}_N^{(-\alpha,-k)}(\Lambda).
\end{equation}
\begin{rem}\label{orthdev} It is worthwhile to point out that
 for  $(\alpha,\beta)\in {}^+\hspace*{-1pt}\Upsilon_1^{\alpha,\beta}\cup {}^+\hspace*{-1pt}\Upsilon_2^{\alpha,\beta},$
we have
\begin{equation}\label{projectorIV}
\big(D_+^{\alpha+l}({}^+\hspace*{-2pt}\pi_N^{(-\alpha,\beta)}u-u),  D^l w_N\big)_{\omega^{(l,\alpha+\beta+l)}}=0,\quad \forall\,
w_N\in  \mathcal P_N,
\end{equation}
for all $l\in {\mathbb N}_0.$
Notice that
$$ ({}^+\hspace*{-2pt}\pi_N^{(-\alpha,\beta)}u-u)(x)=  \sum_{n=N+1}^\infty \hat u_n^{(\alpha,\beta)} \,  {}^+{\hspace*{-3pt}}J_{n}^{(-\alpha,\beta)}(x), $$
and ${\mathcal P}_N={\rm span}\big\{P_n^{(0,\alpha+\beta)}\,:\,0\le n\le N\big\}.$  Using the property
$D^{\alpha+l}_+ =(-1)^lD^l D_+^\alpha,$  we  obtain  \eqref{projectorIV} from  \eqref{derivative+A}, \eqref{derimulti} and
the orthogonality of the classical Jacobi polynomials  (cf. \eqref{jcbiorth}).  \qed
\end{rem}

To characterize the regularity of $u,$ we introduce the non-uniformly weighted  space involving fractional derivatives:
\begin{equation}\label{BmspsA}
{}^+\hspace*{-2pt}{\mathcal B}^{m}_{\alpha,\beta} (\Lambda):=\big\{u\in L^2_{\omega^{(-\alpha,\beta)}}(\Lambda)\,:\, D^{\alpha+l}_+  u\in L_{\omega^{(l,\alpha+\beta+l)}}^2(\Lambda)\;\; {\rm for}\;\;
0\le l\le  m\big\},\;\;\; m\in {\mathbb N}_0.
\end{equation}
By \eqref{D+orthl} and \eqref{uexpansionA} or \eqref{uexpansionAB}, we have  that for $(\alpha,\beta)\in {}^+\hspace*{-1pt}\Upsilon_1^{\alpha,\beta}\cup {}^+\hspace*{-1pt}\Upsilon_2^{\alpha,\beta}$  and $l\in {\mathbb N}_0,$
\begin{equation}\label{Dsnorm2}
\begin{split}
&\big\| D_+^{\alpha+l} u  \big\|_{\omega^{(l,\alpha+\beta+l)}}^2=\sum_{n=\tilde l}^\infty  h_{n, \tilde l}^{(\alpha,\beta)}  \big|\hat u_{n}^{(\alpha,\beta)}\big|^2,
\end{split}
\end{equation}
where $\tilde l=l $  for  $(\alpha,\beta)\in {}^+\hspace*{-1pt}\Upsilon_1^{\alpha,\beta};$
$\tilde l=\max\{l,k\}$ for  $(\alpha,\beta)\in {}^+\hspace*{-1pt}\Upsilon_2^{\alpha,\beta}, $
and  $h_{n, \tilde l}^{(\alpha,\beta)}$  is defined in \eqref{constanth}.

Our main result on the projection errors  for these two cases is stated as follows.
\begin{thm}\label{Th3.1main}  Let $(\alpha,\beta)\in {}^+\hspace*{-1pt}\Upsilon_1^{\alpha,\beta}\cup {}^+ \hspace*{-1pt}\Upsilon_2^{\alpha,\beta}$, and let   $u\in {}^+\hspace*{-2pt}{\mathcal B}^m_{\alpha,\beta} (\Lambda) $ with $m\in {\mathbb N}_0$.
\begin{itemize}
\item For $0\le l\le m\le N,$
\begin{equation}\label{mainest0A}
\begin{split}
\big\|D_+^{\alpha+l}({}^+\hspace*{-2pt} \pi_N^{(-\alpha,\beta)} u-u)\big\|_{\omega^{(l,\alpha+\beta+l)}}
&\le N^{(l-m)/2}
\sqrt{\frac{(N-m+1)!}
{(N-l+1)!}}\,
\big\|D_+^{\alpha+m}  u\big\|_{\omega^{(m,\alpha+\beta+m)}}.
\end{split}
\end{equation}
In particular,  if $m$ is fixed, then
\begin{equation}\label{mainest0Aest}
\begin{split}
\big\|D_+^{\alpha+l}({}^+\hspace*{-2pt} \pi_N^{(-\alpha,\beta)} u-u)\big\|_{\omega^{(l,\alpha+\beta+l)}}
&\le c N^{l-m} \big\|D_+^{\alpha+m}  u\big\|_{\omega^{(m,\alpha+\beta+m)}}.
\end{split}
\end{equation}
\item For $0\le m\le N,$ we also have  the  $L^2_{\omega^{(-\alpha,\beta)}}$-estimates:
\begin{equation}\label{mainest0Anew}
\begin{split}
\big\|{}^+\hspace*{-2pt} \pi_N^{(-\alpha,\beta)} u-u\big\|_{\omega^{(-\alpha,\beta)}}\le
cN^{-\alpha}
\sqrt{\frac{(N-m+1)!}
{(N+m+1)!}}\,
\big\|D_+^{\alpha+m}  u\big\|_{\omega^{(m,\alpha+\beta+m)}}.
\end{split}
\end{equation}
In particular,  if $m$ is fixed, then
\begin{equation}\label{mainest0Anewfixed}
\begin{split}
\big\|{}^+\hspace*{-2pt} \pi_N^{(-\alpha,\beta)} u-u\big\|_{\omega^{(-\alpha,\beta)}}\le
cN^{-(\alpha+m)} \big\|D_+^{\alpha+m}  u\big\|_{\omega^{(m,\alpha+\beta+m)}}.
\end{split}
\end{equation}
\end{itemize}
Here,  $c\approx 1$ for $N\gg1$.
\end{thm}
\begin{proof}
 By \eqref{uexpansionA} (or \eqref{uexpansionAB}),  \eqref{projectorII3} (or \eqref{projectorII})  and
\eqref{Dsnorm2},
\begin{equation}\label{pfwang1}
\begin{split}
\big\|D_+^{\alpha+l}({}^+\hspace*{-2pt} \pi_N^{(-\alpha,\beta)} u-u)\big\|_{\omega^{(l,\alpha+\beta+l)}}^2&=\sum_{n=N+1}^\infty  h_{n, l}^{(\alpha,\beta)}  \big|\hat u_{n}^{(\alpha,\beta)}\big|^2 =    \sum_{n=N+1}^\infty   \frac{h_{n,l}^{(\alpha,\beta)}}{h_{n,m}^{(\alpha,\beta)}} h_{n,m}^{(\alpha,\beta)}\, \big|\hat u_{n}^{(\alpha,\beta)}\big|^2
\\&  \le  \frac{h_{N+1,l}^{(\alpha,\beta)}}
{h_{N+1,m}^{(\alpha,\beta)}}\big\|D_+^{\alpha+m}  u\big\|_{\omega^{(m,\alpha+\beta+m)}}^2.
\end{split}
\end{equation}
We now estimate the  constant factor.
 By  \eqref{anotation}, \eqref{constanth} and a direct calculation, we find that for   $0\le l\le m\le N,$
\begin{equation}\label{pfwang2}
\begin{split}
 \frac{h_{N+1,l}^{(\alpha,\beta)}}
{h_{N+1,m}^{(\alpha,\beta)}}& = \frac{\Gamma(N+\alpha+\beta+l+2)} {\Gamma(N+\alpha+\beta+m+2)} \frac{(N-m+1)!} {(N-l+1)!}\\
&=\frac{1} {(N+\alpha+\beta+2+l)\cdots (N+\alpha+\beta+1+m)} \frac{(N-m+1)!} {(N-l+1)!}\\
&\le N^{l-m}  \frac{(N-m+1)!} {(N-l+1)!},
\end{split}
\end{equation}
where we used the fact: $\alpha+\beta>-1.$ Therefore, the estimate  \eqref{mainest0A} follows from
\eqref{pfwang1}-\eqref{pfwang2} immediately.

 We now   turn to \eqref{mainest0Aest}.  Let us recall the property of the Gamma function (see \cite[(6.1.38)]{Abr.I64}):
\begin{equation}\label{stirlingfor}
\Gamma(x+1)=\sqrt{2\pi}
x^{x+1/2}\exp\Big(-x+\frac{\theta}{12x}\Big),\quad \forall\,
x>0,\;\;0<\theta<1.
\end{equation}
We can show  that   for any constant   $a, b\in {\mathbb R},$   $n\in {\mathbb N}, $  $n+a>1$ and $n+b>1$ (see  \cite[Lemma 2.1]{ZhWX13}),
\begin{equation}\label{Gammaratio}
\frac{\Gamma(n+a)}{\Gamma(n+b)}\le \nu_n^{a,b} n^{a-b},
\end{equation}
where
\begin{equation}\label{ConstUpsilon}
\nu_n^{a,b}=\exp\Big(\frac{a-b}{2(n+b-1)}+\frac{1}{12(n+a-1)}+\frac{(a-b)^2}{n}\Big).
\end{equation}
 Using  the property $\Gamma(n+1)=n!$ and \eqref{Gammaratio}, we find that for $m\le N,$
\begin{equation}\label{ratiovalue2}
\begin{split}
\frac{(N-m+1)!} {(N-l+1)!}
&\le \nu_{N}^{2-m,2-l} N^{l-m},
\end{split}
\end{equation}
where $\nu_{N}^{2-m,2-l}\approx 1$ for  fixed $m$ and $N\gg 1.$ Thus, we obtain \eqref{mainest0Aest}
from   \eqref{mainest0A}  immediately.

The $L^2_{\omega^{(-\alpha,\beta)}}$-estimates can be obtained by using the same argument. We sketch  the derivation below.    By \eqref{uexpansionB}  and
\eqref{Dsnorm2},
\begin{equation}\label{pfwang1e}
\begin{split}
 \big\|{}^+ \hspace*{-2pt}\pi_N^{(-\alpha,\beta)} u &-u\big\|_{\omega^{(-\alpha,\beta)}}^2=
\sum_{n=N+1}^\infty  \gamma_n^{(\alpha,\beta)}  \, \big| \hat u_{n}^{(\alpha,\beta)}\big|^2 \\
& \le \frac{\gamma_{N+1}^{(\alpha,\beta)}}
{h_{N+1,m}^{(\alpha,\beta)}}    \sum_{n=N+1}^\infty   h_{n,m}^{(\alpha,\beta)} \, \big|\hat u_{n}^{(\alpha,\beta)}\big|^2
 \le  \frac{\gamma_{N+1}^{(\alpha,\beta)}}
{h_{N+1,m}^{(\alpha,\beta)}}\big\|D_+^{\alpha+m}  u\big\|_{\omega^{(m,\alpha+\beta+m)}}^2.
\end{split}
\end{equation}
Working out the constants by \eqref{co-gamma} and \eqref{constanth}, we use  \eqref{Gammaratio} again to get that
\begin{equation}\label{ratiovalue2e}
\begin{split}
\frac{\gamma_{N+1}^{(\alpha,\beta)}}
{h_{N+1,m}^{(\alpha,\beta)}}&=\frac{\Gamma(N+\beta+2)} {\Gamma(N+\alpha+2)}
\frac{\Gamma(N+m+2)} {\Gamma(N+\alpha+\beta+m+2)}\frac{(N-m+1)!}{(N+m+1)!}\\
&\le \nu_{N}^{\alpha+2,\beta+2} N^{\beta-\alpha} \;  \nu^{2,\alpha+\beta+2}_N (N+m)^{-(\alpha+\beta)}  \frac{(N-m+1)!}{(N+m+1)!}\\
&\le c N^{-2(\alpha+m)}\;\; \text{(if  $m$ is fixed)}.
\end{split}
\end{equation}
This ends the proof.
 \end{proof}

\begin{rem}\label{remoptimal} We see from  the  above estimates  that
optimal order of convergence can be attained for approximation of $u$ by
 its orthogonal projection ${}^+ \hspace*{-2pt}\pi_N^{(-\alpha,\beta)} u$ in  both
 $L^2_{\omega^{(-\alpha,\beta)}}(\Lambda)$ and ${}^+\hspace*{-2pt}{\mathcal B}^l_{\alpha,\beta} (\Lambda),$
 when $u$ belongs to a properly weighted space involving  proper orders of fractional derivatives.
 \qed
 \end{rem}

\subsubsection{Case II: $(\alpha,\beta)\in {}^+\hspace*{-1pt}\Upsilon_3^{\alpha,\beta}$.}  In this case, the main difficulty resides in  that the GJFs $\{{}^+{\hspace*{-3pt}}J_{n}^{(-\alpha,\beta)}\}$ are no longer  orthogonal on $\Lambda$.  Thus, we adopt  a  different route  to derive the approximation results.
For any $u$ such that $D_+^\alpha u\in L^2_{\omega^{(0,\alpha+\beta)}}(\Lambda),$ it admits the following  unique expansion:
\begin{equation}\label{D+exp}
D_+^\alpha u(x)=\sum_{n=0}^\infty  \hat v_n^{(0,\alpha+\beta)}  P_n^{(0,\alpha+\beta)}(x),
\end{equation}
where by \eqref{jcbiorth},
\begin{equation}\label{casesA}
\hat v_n^{(0,\alpha+\beta)}=
\frac 1 {\gamma_n^{(0,\alpha+\beta)}} \int_{-1}^1 D_+^\alpha u(x)\;  P_n^{(0,\alpha+\beta)}(x) \, (1+x)^{\alpha+\beta}\, dx.
\end{equation}
From the definition of the  usual orthogonal projection operator $\Pi_N^{(0,\alpha+\beta)}:\;L^2_{\omega^{(0,\alpha+\beta)}}(\Lambda) \to {\mathcal P}_N $, we have
\begin{equation}\label{D+exp2proj}
\big(\Pi_N^{(0,\alpha+\beta)}(D_+^\alpha u)-D^\alpha_+ u, v_N\big)_{\omega^{(0,\alpha+\beta)}}=0,\quad \forall\,v_N\in {\mathcal P}_N,
\end{equation}
and
\begin{equation}\label{D+exp2a}
\Pi_N^{(0,\alpha+\beta)}(D_+^\alpha u)(x)=\sum_{n=0}^N  \hat v_n^{(0,\alpha+\beta)}  P_n^{(0,\alpha+\beta)}(x).
\end{equation}
Let ${}^+\hspace*{-2pt} {\mathcal F}_N^{(-\alpha,\beta)}(\Lambda)$ be the finite-dimensional space as defined in   \eqref{frac-polysps} but for  $(\alpha,\beta)\in \Upsilon_3^{\alpha,\beta}$.

\begin{lmm}\label{ulem} Let $\alpha>0, -\alpha-1<\beta<-1$ and $-\beta\not \in {\mathbb N}.$  For   any $u$ such that $D_+^\alpha u\in L^2_{\omega^{(0,\alpha+\beta)}}(\Lambda),$ there exists a unique $ u_N=:{}^+\hspace*{-2pt}\pi_N^{(-\alpha,\beta)}u\in  {}^+\hspace*{-2pt} {\mathcal F}_N^{(-\alpha,\beta)}(\Lambda)$ such that
\begin{equation}\label{unqiueA}
\Pi_N^{(0,\alpha+\beta)}(D_+^\alpha u)(x)=D_+^\alpha ({}^+\hspace*{-2pt}\pi_N^{(-\alpha,\beta)}u)(x),
\end{equation}
and
\begin{equation}\label{D+exp2projB}
\big(D_+^\alpha ({}^+\hspace*{-2pt}\pi_N^{(-\alpha,\beta)}u-u), v_N\big)_{\omega^{(0,\alpha+\beta)}}=0,\quad \forall\,v_N\in {\mathcal P}_N.
\end{equation}
\end{lmm}
\begin{proof}  From  the expansion coefficients $\big\{ \hat v_n^{(0,\alpha+\beta)}\big\}_{n=0}^N$  in \eqref{casesA}, we construct
\begin{equation}\label{pNexp}
u_N(x)=\sum_{n=0}^N  \frac{n!\,  \hat v_n^{(0,\alpha+\beta)}}{\Gamma(n+\alpha+1)}  {}^+{\hspace*{-3pt}}J_{n}^{(-\alpha,\beta)}(x)\in  {}^+\hspace*{-2pt} {\mathcal F}_N^{(-\alpha,\beta)}(\Lambda).
\end{equation}
Acting $D_+^\alpha$  on both sides, we obtain from  \eqref{derivative+A} and \eqref{D+exp2a} that
$$
D_+^\alpha u_N(x)= \sum_{n=0}^N  \hat v_n^{(0,\alpha+\beta)}  P_n^{(0,\alpha+\beta)}(x)=\Pi_N^{(0,\alpha+\beta)} (D_+^\alpha u)(x).
$$
Note that the expansion in  \eqref{casesA} is unique, so we specifically denote $u_N$ by ${}^+\hspace*{-2pt}\pi_N^{(-\alpha,\beta)}u,$ and \eqref{unqiueA} is shown.

The property  \eqref{D+exp2projB} is a direct consequence of \eqref{D+exp2proj} and \eqref{unqiueA}.
\end{proof}

With Lemma \ref{ulem} at our disposal, we can obtain the following error estimates.
\begin{thm}\label{Th3.3main}  Let $\alpha>0, -\alpha-1<\beta<-1$ and $-\beta\not \in {\mathbb N},$ and let  ${}^+\hspace*{-1pt}\pi_N^{(-\alpha,\beta)}$ be defined
in Lemma {\rm \ref{ulem}}.  Suppose that $D_+^{\alpha+l} u\in L^2_{\omega^{(l,\alpha+\beta+l)}}(\Lambda)$ with $0\le l\le m\le N.$ Then we have
\begin{equation}\label{newest}
\begin{split}
\big\|D_+^\alpha ({}^+\hspace*{-2pt}\pi_N^{(-\alpha,\beta)}u-u)\big\|_{\omega^{(0,\alpha+\beta)}}
& \le
N^{-m/2}\sqrt{\frac{(N-m+1)!}{(N+1)!}} \big\|D_+^{\alpha+m} u\big\|_{\omega^{(m,\alpha+\beta+m)}}.
\end{split}
\end{equation}
In particular, if $m$ is fixed, we have
\begin{equation}\label{newest2}
\begin{split}
\big\|D_+^\alpha ({}^+\hspace*{-2pt}\pi_N^{(-\alpha,\beta)}u-u)\big\|_{\omega^{(0,\alpha+\beta)}}
&\le cN^{-m} \big\|D_+^{\alpha+m} u\big\|_{\omega^{(m,\alpha+\beta+m)}},
\end{split}
\end{equation}
where  the constant $c\approx 1$ for $N\gg 1.$
\end{thm}
\begin{proof} Using the relation \eqref{unqiueA}, we further derive from \eqref{jcbiorth},  \eqref{D+exp} and  \eqref{D+exp2a} that
\begin{equation}\label{D+estform}
\begin{split}
\big\|D_+^\alpha ({}^+\hspace*{-2pt}\pi_N^{(-\alpha,\beta)}u-u)\big\|_{\omega^{(0,\alpha+\beta)}}^2&=\big\|\Pi_N^{(0,\alpha+\beta)}(D_+^\alpha u)- D_+^\alpha u\big\|_{\omega^{(0,\alpha+\beta)}}^2\\&=\sum_{n=N+1}^\infty \gamma_n^{(0,\alpha+\beta)} \big|\hat v_n^{(0,\alpha+\beta)}\big|^2.
\end{split}
\end{equation}
We find from \eqref{jcbiorth}, \eqref{derimulti} and \eqref{D+exp}  that
\begin{equation}\label{anothorth}
\big\|D_+^{\alpha+m} u\big\|^2_{\omega^{(m,\alpha+\beta+m)}}=\sum_{n=m}^\infty \mu_{n,m}^{(0,\alpha+\beta)} \big|\hat v_n^{(0,\alpha+\beta)}\big|^2,
\end{equation}
where for $n\ge m, $
\begin{equation}\label{muconst}
\mu_{n,m}^{(0,\alpha+\beta)}= \big(\kappa_{n,m}^{(0,\alpha+\beta)} \big)^2 \gamma_{n-m}^{(m,\alpha+\beta+m)}= \frac{2^{\alpha+\beta+1} n!\, \Gamma(n+\alpha+\beta+m+1)}
{(2n+\alpha+\beta+1)\,(n-m)!\, \Gamma(n+\alpha+\beta+1)}.
\end{equation}
In view of the above facts, we  work out the constants by using \eqref{co-gamma} and obtain
\begin{equation}\label{D+estform3}
\begin{split}
 \big\|D_+^\alpha  ({}^+\hspace*{-2pt}\pi_N^{(-\alpha,\beta)}u &-u)\big\|_{\omega^{(0,\alpha+\beta)}}^2
\le \frac{\gamma_{N+1}^{(0,\alpha+\beta)}}{\mu_{N+1,m}^{(0,\alpha+\beta)} }  \big\|D_+^{\alpha+m} u\big\|^2_{\omega^{(m,\alpha+\beta+m)}}\\
& =\frac 1 {(N+\alpha+\beta+2)_m} \frac{(N+1-m)!} {(N+1)!} \big\|D_+^{\alpha+m} u\big\|^2_{\omega^{(m,\alpha+\beta+m)}}\\
& \le N^{-m}  \frac{(N+1-m)!} {(N+1)!} \big\|D_+^{\alpha+m} u\big\|^2_{\omega^{(m,\alpha+\beta+m)}}.
\end{split}
\end{equation}
This yields \eqref{newest}.  For fixed $m,$  we apply \eqref{Gammaratio} to deal with  the above factorials and derive  \eqref{newest2} immediately.
\end{proof}

\subsection{Approximation results for GJFs $\big\{{}^- {\hspace*{-3pt}}J_{n}^{(\alpha,-\beta)}\big\}$}
 The results  established in the previous subsection can be extended to  $\big\{{}^- {\hspace*{-3pt}}J_{n}^{(\alpha,-\beta)}\big\}$ straightforwardly, thanks to \eqref{parity2}. Below, we sketch  the corresponding notation and results.

 Define the parameter  set
\begin{equation}\label{salphacondB}
{}^-\hspace*{-1pt}\Upsilon^{\alpha,\beta}:=\big\{(\alpha,\beta)\,:\, \beta> 0,\;\;  \alpha+\beta>-1\big\},
\end{equation}
 which we   split into three disjoint subsets:
\begin{equation}\label{3subcaseB}
\begin{split}
&  {}^-\hspace*{-1pt}\Upsilon_1^{\alpha,\beta}:= \big\{(\alpha,\beta)\,:\, \beta> 0,\; \alpha>-1\big\};\\
& {}^-\hspace*{-1pt}\Upsilon_2^{\alpha,\beta}:= \big\{(\alpha,\beta)\,:\, \beta> 0,\; -\beta-1<\alpha=-k\le -1,\;  k\in {\mathbb N}\big\};\\
& {}^-\hspace*{-1pt}\Upsilon_3^{\alpha,\beta}:= \big\{(\alpha,\beta)\,:\, \beta> 0,\; -\beta-1<\alpha<-1,\;  -\alpha\not\in {\mathbb N}\big\}.
\end{split}
\end{equation}

Consider the $L^2_{\omega^{(\alpha,-\beta)}}$-orthogonal projection: ${}^-\hspace*{-2pt}\pi_N^{(\alpha,-\beta)}u\in {}^-\hspace*{-2pt} {\mathcal F}_N^{(\alpha,-\beta)}(\Lambda)$
for $(\alpha,\beta)\in {}^-\hspace*{-1pt}\Upsilon_1^{\alpha,\beta}\cup {}^-\hspace*{-1pt}\Upsilon_2^{\alpha,\beta},$
where the notation is defined in a fashion similar to that in the previous subsection.    In this context, we define
\begin{equation}\label{BmspsB}
{}^-\hspace*{-2pt}{\mathcal B}^{m}_{\alpha,\beta} (\Lambda):=\big\{u\in L^2_{\omega^{(\alpha,-\beta)}}(\Lambda)\,:\, D^{\beta+l}_-  u\in L_{\omega^{(\alpha+\beta+l,l)}}^2(\Lambda)\;\; {\rm for}\;\;
0\le l\le  m\big\},\quad m\in {\mathbb N}_0.
\end{equation}
Following the argument as in the proof of Theorem \ref{Th3.1main}, we can derive the following error  estimates.
\begin{thm}\label{Thm3.2} Let $(\alpha,\beta)\in {}^-\Upsilon_1^{\alpha,\beta}\cup {}^-\Upsilon_2^{\alpha,\beta}$, and let   $u\in {}^-\hspace*{-1pt}{\mathcal B}^m_{\alpha,\beta} (\Lambda) $ with $m\in {\mathbb N}_0$. 
\begin{itemize}
\item For $0\le l\le m\le N,$
\begin{equation}\label{mainest0B}
\begin{split}
\big\|D_-^{\beta+l}({}^-\hspace*{-2pt} \pi_N^{(\alpha,-\beta)} u-u)\big\|_{\omega^{(\alpha+\beta+l,l)}}&\le N^{(l-m)/2}
\sqrt{\frac{(N-m+1)!}{(N-l+1)!}}\,\big\|D_-^{\beta+m}  u\big\|_{\omega^{(\alpha+\beta+m,m)}}.
\end{split}
\end{equation}
In particular,  if $m$ is fixed, we have
\begin{equation}\label{mainest0Best}
\begin{split}
\big\|D_-^{\beta+l}({}^-\hspace*{-2pt} \pi_N^{(\alpha,-\beta)} u-u)\big\|_{\omega^{(\alpha+\beta+l,l)}}&\le cN^{l-m}\,\big\|D_-^{\beta+m}  u\big\|_{\omega^{(\alpha+\beta+m,m)}}.
\end{split}
\end{equation}
\item For $0\le m\le N,$ we also have  the  $L^2_{\omega^{(\alpha,-\beta)}}$-estimates:
\begin{equation}\label{mainest0AnewB}
\begin{split}
\big\|{}^-\hspace*{-2pt} \pi_N^{(\alpha,-\beta)} u-u\big\|_{\omega^{(\alpha,-\beta)}}\le
cN^{-\beta}
\sqrt{\frac{(N-m+1)!}
{(N+m+1)!}}\,
\big\|D_-^{\beta+m}  u\big\|_{\omega^{(\alpha+\beta+m,m)}}.
\end{split}
\end{equation}
In particular,  if $m$ is fixed, then
\begin{equation}\label{mainest0AnewfixedB}
\begin{split}
\big\|{}^-\hspace*{-2pt} \pi_N^{(\alpha,-\beta)} u-u\big\|_{\omega^{(\alpha,-\beta)}}\le
cN^{-(\beta+m)} \big\|D_-^{\beta+m}  u\big\|_{\omega^{(\alpha+\beta+m,m)}}.
\end{split}
\end{equation}
\end{itemize}
Here, $c\approx 1$ for $N\gg1$.
\end{thm}

Next, we consider $(\alpha,\beta)\in {}^-\hspace*{-1pt}\Upsilon_3^{\alpha,\beta}$.
For $D_-^\beta u\in L^2_{\omega^{(\alpha+\beta,0)}}(\Lambda),$  we  define the  operator ${}^-\hspace*{-1pt} \Pi_N^{(\alpha,-\beta)}$
similarly as that in Lemma \ref{ulem}. Following the lines as in the proof  of Theorem \ref{Th3.3main},  we can obtain the following estimates.
\begin{thm}\label{Th3.3mainB}  Let  $(\alpha,\beta)\in {}^-\hspace*{-1pt}\Upsilon_3^{\alpha,\beta}$.  Suppose that $D_-^{\beta+l} u\in L^2_{\omega^{(\alpha+\beta+l,l)}}(\Lambda)$ with $0\le l\le m\le N.$ Then we have
\begin{equation}\label{newestB}
\begin{split}
\big\|D_-^\beta ({}^-\hspace*{-2pt}\pi_N^{(\alpha,-\beta)}u-u)\big\|_{\omega^{(\alpha+\beta,0)}}
& \le
N^{-m/2}\sqrt{\frac{(N-m+1)!}{(N+1)!}} \big\|D_-^{\beta+m} u\big\|_{\omega^{(\alpha+\beta+m,m)}}.
\end{split}
\end{equation}
In particular, if $m$ is fixed, we have
\begin{equation}\label{newest2B}
\begin{split}
\big\|D_-^\beta ({}^-\hspace*{-2pt}\pi_N^{(\alpha,-\beta)}u-u)\big\|_{\omega^{(\alpha+\beta,0)}}
&\le cN^{-m} \big\|D_-^{\beta+m} u\big\|_{\omega^{(\alpha+\beta+m,m)}},
\end{split}
\end{equation}
where  the constant $c\approx 1$ for $N\gg 1.$
\end{thm}

\begin{rem}\label{betterunderstanding}
To have  a better understanding of the above approximation results,  we compare GJF and Legendre approximation  to the function:
 \begin{equation}\label{signfun}
 u(x)=(1+x)^{b} g(x),\quad b\in {\mathbb R}^+,\;\; x\in \Lambda,
 \end{equation}
 where  $g$ is analytic within a domain containing $\Lambda.$  Recall the best $L^2$-approximation of $u$ by its orthogonal projection $\pi_N^Lu$ (see, e.g., \cite[Ch. 3]{ShenTangWang2011}):
 $$\|\pi_N^L u-u\|\le c N^{1-m}\|D^mu\|_{\omega^{(m,m)}}.$$
 If $b$ is non-integer,  a direct calculation shows that $u$ has a limited regularity: $m<1+2b-\epsilon$ for small $\epsilon>0,$ in this usual weighted norm involving ordinary derivatives.
 We now consider GJF approximation \eqref{mainest0AnewB} to $u$ in  \eqref{signfun}. 
   Using the explicit formulas for fractional integral/derivative of $(1+x)^b$ and the  Leibniz' formula (see \cite[Ch. 2]{Diet10}), we find that if  $\beta=b,$ $D_-^{\beta+m}u$ is analytic as well for any $m\in {\mathbb N}_0,$ so by
   \eqref{signfun}  with $\alpha=0, \beta=b$ and $m=N,$ and using \eqref{stirlingfor}, we have
   $$ \big\|{}^-\hspace*{-2pt} \pi_N^{(0,-\beta)} u-u\big\|\le \big\|{}^-\hspace*{-2pt} \pi_N^{(0,-\beta)} u-u\big\|_{\omega^{(0,-\beta)}}\le
cN^{-(\beta+1/4)}\Big(\frac e {2N}\Big)^N \big\|D_-^{\beta+N}  u\big\|_{\omega^{(\beta+N,N)}}.$$
 This implies the exponential convergence $O(e^{-cN}).$

 Also note that if $u$ is smooth, e.g.,  $b\in {\mathbb N},$  we can only get a limited convergence rate by choosing a non-integer $\beta.$  Indeed,  a direct calculation by using the formulas in \cite{Diet10} yields
 $$D_-^{\beta+m}u=(1+x)^{b-\beta-m} h(x),$$
 where $h$ is analytic.  Therefore, we have that $\big\|D_-^{\beta+m}  u\big\|_{\omega^{(\beta+m,m)}}<\infty$ only when $m+2\beta<1+2b-\epsilon.$
 \qed
 \end{rem}

\section{Applications to fractional differential equations}\label{sect:appl}
\setcounter{equation}{0}
\setcounter{lmm}{0}
\setcounter{thm}{0}

It is well-known that the underlying solution  of a FDE usually exhibits singular behaviors at the boundaries, even when the given data are regular.  Accordingly,  the solution and data are not  always  in the same types of Sobolev spaces as opposite to  DEs of integer derivatives.  Hence,   the  use of polynomial approximations can only lead to limited convergence rate.
In this section,  we shall construct Petrov-Galerkin spectral methods using  GJFs as basis functions for several prototypical FDEs, and
demonstrate that
\begin{itemize}
 \item[(i)]   The convergence rate of our approach   only
depends on the regularity of the data in the usual weighted Sobolev space, regardless of the singular behavior of their solutions, so truly spectral accuracy can be achieved,  if the input of a FDE  is smooth enough.
\item[(ii)]  With a suitable choice of the parameters in the GJF basis,  the resulted linear systems are usually  sparse and sometimes  diagonal.
\end{itemize}
We shall provide ample numerical results to validate the theoretical analysis.
We believe that the study of  these prototypical FDEs can shed light on the investigation of more complicated FDEs.

\subsection{Fractional initial value problems (FIVPs)}
As a first example, we consider the factional initial value problem  of order $s\in (k-1, k)$ with $k\in {\mathbb  N}:$
\begin{equation}\label{prob11}
\begin{split}
  & D_{+}^s  u(x)=f(x), \;\;  x\in {\Lambda};\quad u^{(l)}(1)=0,\quad l=0, \cdots, k-1,    
  \end{split}
 \end{equation}
 where  $f\in L^2(\Lambda).$  
The GJF-spectral-Petrov-Galerkin scheme is to  find $u_N\in {}^+\hspace*{-2pt} {\mathcal F}_N^{(-s,-s)}(\Lambda) $ (defined in   \eqref{frac-polysps}) such that
 \begin{equation}\label{galformA}
 (D_{+}^s  u_N,  v_N)=(f,   v_N),\quad \forall\, v_N\in {\mathcal P}_{N}.     
 \end{equation}
  Using the GJF basis, we can write
 \begin{equation}\label{errorest2A}
 u_N(x)=\sum_{ n=0}^{ N} \tilde u_{n}^{(s)}\,   {}^+ {\hspace*{-3pt}}J_{n}^{(-s,-s)}(x)\in {}^+\hspace*{-2pt} {\mathcal F}_N^{(-s,-s)}(\Lambda) .
 \end{equation}
 Taking $v_N= P_k$ in \eqref{galformA},  we derive from   \eqref{derivative+A} and the orthogonality of Legendre polynomials  that
\begin{equation}\label{coefeqns2A}
\tilde  u_{n}^{(s)}= \frac{n!} {\Gamma(n+s+1)}   \tilde  f_{n},\quad 0\le   n\le N,
\end{equation}
where $\tilde  f_{n}$ is the $(n+1)$-th  coefficient of the Legendre expansion of $f$. 
Therefore, we obtain the numerical solution $u_N$ by inserting \eqref{coefeqns2A} into \eqref{errorest2A}.

The following error estimate shows the spectral accuracy of this GJF-Petrov-Galerkin approximation.
\begin{thm}\label{Th3141}  Let  $u$ and $u_N$ be the solution  of \eqref{prob11} and
\eqref{galformA}, respectively.
If  
 $f^{(l)}\in L^2_{\omega^{(l,l)}}(\Lambda)$ for all $0\le l\le m,$   then  we have that for   $0\le m\le N,$
\begin{equation}\label{mainest00A2A}
\|D_+^{s} (u-u_N)\|\leq  cN^{-m} \|f^{(m)}\|_{\omega^{(m,m)}},
\end{equation}
where $c$ is a positive constant independent of $u, N$ and $m.$
\end{thm}
\begin{proof}
Let ${}^+\hspace*{-2pt}\pi_N^{(-s,-s)}u$ be the same as in \eqref{projectorII} for $0<s<1,$ and as in
\eqref{D+exp2projB} for $s>1,$ respectively.  By \eqref{projectorIV} (with $l=0$)  and  \eqref{D+exp2projB}, we have
$$
\big(D_{+ }^s ({}^+\hspace*{-2pt} \pi_N^{(-s,-s)} u-u), \psi\big)=0,\quad \forall\, \psi\in {\mathcal P}_N.
$$
Then by  \eqref{prob11},
\begin{equation}\label{Rela-Pnu-Pnf}
\begin{split}
\big(f- D_{+ }^s {}^+ \hspace*{-2pt}\pi_N^{(-s,-s)}u,\, &\psi\big)
=\big(D_{+ }^s u-D_{+ }^s  {}^+ \hspace*{-2pt}\pi_N^{(-s,-s)} u, \psi\big)=0, \quad \forall\, \psi\in {\mathcal P}_N.
\end{split}
\end{equation}
Let $\pi_N f$ be the $L^2$-orthogonal projection of $f$ upon ${\mathcal P}_N.$    We infer from
\eqref{Rela-Pnu-Pnf} that  $D_{+ }^s  {}^+ \hspace*{-2pt}\pi_N^{(-s,-s)} u=\pi_N f.$  On the other hand,  by
\eqref{galformA},  $D_{+ }^s  u_N=\pi_N f.$ Thus, we have $D_+^{s} ({}^+\hspace*{-2pt} \pi_N^{(-s,-s)} u-u_N)=0.$
Therefore,  it follows from Theorem \ref{Th3.1main} (with $\alpha=-\beta=s$ and $0<s<1$), Theorem \ref{Th3.3main} (with $\alpha=-\beta=s$ and $s>1$) that
\begin{equation*}\label{mainest00A2A0}
\|D_+^{s} (u-u_N)\|=\big\|D_+^{s} (u-{}^+\hspace*{-2pt} \pi_N^{(-s,-s)} u)\big\|\leq  cN^{-m}\|D_+^{s+m}  u\big\|_{\omega^{(m,m)}}\le  c N^{-m}\|f^{(m)}\|_{\omega^{(m,m)}}.
\end{equation*}
This ends the proof. 
\end{proof}

\begin{rem}\label{moregenform} One can also construct  a similar Petrov-Galerkin scheme  for the following  more general FIVPs of order $s\in (k-1, k)$ with $k\in {\mathbb  N}:$
\begin{equation}\label{prob112A}
\begin{split}
  & {\mathcal L}[u]:= D_{+}^s  u(x)+p_1(x) D_+^{s-1}u(x)+\cdots+p_{k-1}(x)D_+^{s-k+1} u(x)=f(x), \;\;  x\in {\Lambda};\\
  & u^{(l)}(1)=0,\quad l=0, \cdots, k-1,    
  \end{split}
 \end{equation}
 where  
 $f$ and $\{p_j\}$ are continuous functions on $\bar \Lambda.$
 We find from \eqref{derivative+}  that  ${\mathcal L}[{}^+ {\hspace*{-3pt}}J_{n}^{(-s,s)}]$ is a combination of
 products of $p_j$ and polynomials. Hence, one can derive spectrally accurate error estimates as in Theorem \ref{Th3141}.
 If $\{p_j\}$ are constants, the corresponding linear system will be sparse; for general $\{p_j\}$,  one can use a  preconditioned iterative algorithm  as in the integer $s$ case by using the problem with suitable constant constants as a preconditioner (cf.  \cite{ShenTangWang2011}).   \qed
\end{rem}


\subsection{Fractional boundary value problems (FBVPs) }  In accordance with usual BVPs,
it is necessary to classify a FBVP of order $\nu$  as even or odd order as follows.
\begin{itemize}
\item  If $\nu=s+k$ with $s\in (k-1,k)$ and $k\in {\mathbb N}$,  we say  it is  of even order.  In this case, $2k$ boundary conditions should be imposed.
\item  If $\nu=s+k$ with $s\in (k,k+1)$ and $k\in {\mathbb N}$, we say it is  odd order.  In this case, $2k+1$ boundary conditions should be imposed.
\end{itemize}
In practice, the boundary conditions can be of integral type or usual Dirichlet  type,
which oftentimes  lead to different singular behaviour of the solution and should be treated quite differently.
For easy of accessibility, we first consider FBVPs with integral boundary conditions (BCs), and then discuss the more complicated Dirichlet BCs.

 \subsubsection{FBVPs  with integral BCs}
 To fix the idea, we consider the fractional boundary value problem of order $\nu \in (1,2)$: 
\begin{equation}\label{prob12AA}
\begin{split}
  & D_{+}^{\nu}  u(x)=f(x), \;\;  x\in {\Lambda};\quad I_+^{\mu}u(\pm 1)=0,
  \end{split}
 \end{equation}
 where $\mu:=2-\nu\in (0,1),$  and  $I_+^{\mu}$ is the fractional integral operator defined in \eqref{leftintRL}. Here,   $f(x)$ is a given  function with regularity to be specified later. 

   Let $H^1_0(\Lambda)=\{u\in H^1(\Lambda)\,:\, u(\pm 1)=0\},$ and
 $H^{-1}(\Lambda)$ be its dual space.
  Using the property: $D_{+}^{\nu}=D^2 I_+^\mu$ (see  \eqref{ImportRela} ),  we can  formulate the weak form of \eqref{prob12A} as:  Find $v:=I_+^\mu u \in H^1_0(\Lambda)$ such that
  \begin{equation}\label{weakform2A}
 (Dv , \, D  w)=(f,  w),\quad \forall\, w\in H^1_0(\Lambda).
 \end{equation}
 It is well-known that for any $f\in H^{-1}(\Lambda),$ it admits a unique solution  $v\in H^1_0(\Lambda).$ Then  we can recover
  $u$ uniquely from $u=D^{\mu}_+ v,$ thanks to \eqref{rulesa}.

As already mentioned, it is important  to understand the singular behavior of the solution  so as to compass the choice of the parameter that can match the singularity.
For this purpose,  we act $I_+^2$ on  both sides of  \eqref{prob12AA} and impose the boundary conditions, leading to
\begin{equation}\label{Imuux}
I_+^\mu u(x)=-I_+^2 f(x)+\frac{I_+^2 f(-1)}{2}(1-x).
\end{equation}
Thus, by \eqref{rulesa},
\begin{equation}\label{uxux}
u(x)=D_+^\mu I_+^\mu u(x)=-I_+^\nu f(x)+\frac{I_+^2 f(-1)}{2\Gamma(2-\mu)}(1-x)^{1-\mu}.
\end{equation}
Correspondingly, we define the finite-dimensional fractional-polynomial solution space:
 \begin{equation}\label{approxVN}
  V_N:=\big\{ \phi=(1-x)^{1-\mu}\psi\,: \psi\in {\mathcal P_{N-1}} \;\; \text{such that}\;\; I_+^{\mu} \phi(-1)=0\big\}.
 \end{equation}
The GJF-Petrov-Galerkin  approximation is to find $u_N\in V_N$  such that
  \begin{equation}\label{galerkinform2A}
 (D_+^{1-\mu} u_N, \, D  w_N)=(f,  w_N),\quad \forall w_N\in {\mathcal P}_N^0:= {\mathcal  P}_N\cap H^1_0(\Lambda).
 \end{equation}

 In terms of error analysis,  it is more convenient to formulate  \eqref{galerkinform2A} into an equivalent Galerkin approximation (see \eqref{schemeform2A} below).
Indeed,  note that
 \begin{equation}\label{PNspan}
 {\mathcal P}_N={\rm span}\big\{P_n^{(1-\mu,\mu-1)}\,:\,0\le n\le N\big\},
 \end{equation}
and by  \eqref{newbateman} with $\rho=\mu,\alpha=1-\mu$ and $\beta=\mu-1,$  
\begin{equation}\label{Gjcbiform}
\begin{split}
I_+^\mu\, {}^+ {\hspace*{-3pt}}J_{n}^{(\mu-1,\mu-1)}(x)= \frac{\Gamma(n+2-\mu)}{(n+1)!}  {}^+ {\hspace*{-3pt}} J_n^{(-1,-1)}(x)=
\frac{\Gamma(n+2-\mu)}{n!}\int_x^1 P_n(y) dy,
\end{split}
\end{equation}
where we used the formula derived from integrating the Sturm-Liouville equation of Legendre polynomials and using \eqref{derimulti}:
\begin{equation}\label{SLprbform}
I_+^1 P_n(x)=\int_x^1 P_n(y) dy=\frac 1 {2n}(1-x^2) P_{n-1}^{(1,1)}(x)=\frac 1 {n+1}   {}^+ {\hspace*{-3pt}} J_n^{(-1,-1)}(x),\;\; n\ge 1.
\end{equation}
Since for $n\ge 1,  I_+^\mu {}^+ {\hspace*{-3pt}}J_{n}^{(\mu-1)}(\pm 1)=0,$  we have
\begin{equation}\label{VNspan}
 V_N={\rm span}\big\{{}^+ {\hspace*{-3pt}}J_{n}^{(\mu-1,\mu-1)} \,:\,1\le n\le N-1\big\}; \;\;\; {\mathcal P}_N^0={\rm span}
 \big\{I_+^1 P_n\,:\, 1\le n\le N-1\big\}.
 \end{equation}
Thus, we infer from \eqref{Gjcbiform} that the operator $I_+^\mu$ is an isomorphism between $V_N$ and  ${\mathcal P}_N^0.$  Then we can  equivalently  formulate   \eqref{galerkinform2A} as: Find $v_N:=I_+^\mu u_N\in {\mathcal P}_N^0$ such that
 \begin{equation}\label{schemeform2A}
 (Dv_N, \, D  w_N)=(f,  w_N),\quad \forall w_N\in {\mathcal P}_N^0,
 \end{equation}
 which admits a unique solution as with \eqref{weakform2A}.  
 In fact, this formulation facilitates the error analysis, which can be complished by  a standard argument.
  \begin{thm}\label{convresult} Let $u$ and $u_N$ be the solution  of   \eqref{weakform2A} and \eqref{schemeform2A}, respectively. If $I_+^\mu u\in H^1_0(\Lambda)$ and $(1-x^2)^{(m-1)/2}D^{m-\mu}_+u \in L^2(\Lambda)$ with $m\in {\mathbb N},$ then we have
  \begin{equation}\label{2ndforderconvergence}
\|D^{1-\mu}_+(u-u_N)\|\le c N^{1-m} \|D^{m-\mu}_+ u\|_{\omega^{(m-1,m-1)}}.
\end{equation}
In particular, if $f^{(m-2)}\in L^2_{\omega^{(m-1,m-1)}}(\Lambda)$ with $m\ge 2,$ we have
  \begin{equation}\label{2ndforderconvergenceC}
\|D^{1-\mu}_+(u-u_N)\|\le c N^{1-m} \|f^{(m-2)}\|_{\omega^{(m-1,m-1)}}.
\end{equation}
Here, $c$ is a positive constant independent of $N$ and $u.$
\end{thm}
\begin{proof}  Using a standard argument for error analysis of Galerkin approximation, we find from
\eqref{weakform2A} and \eqref{schemeform2A} that
\begin{equation}\label{proof1}
\|D(v-v_N)\|=\inf_{v_N^*\in {\mathcal P}_N^0} \|D(v-v_N^*)\|.
\end{equation}
Let $\pi_N^{1,0}$ be the  usual $H^1_0$-orthogonal projection upon ${\mathcal P}_N^0,$ and recall the approximation result (see e.g., \cite[Ch. 3]{ShenTangWang2011}):
\begin{equation}\label{approxresult}
\|D(v-\pi_N^{1,0}v)\|\le c N^{1-m}\|D^m v\|_{\omega^{(m-1,m-1)}}.
\end{equation}
Recall that $v=I_+^\mu u$ and $v_N=I_+^\mu u_N,$ so we  take $v_N^*=\pi_N^{1,0}v$ in  \eqref{proof1} and obtain the desired  estimate \eqref{2ndforderconvergence} from \eqref{approxresult}.

The estimate \eqref{2ndforderconvergenceC} follows immediately from \eqref{prob12AA} and \eqref{2ndforderconvergence} by noting that $D^{m-2} f=D_+^{m-\mu} u.$
\end{proof}

 Now, we briefly describe the implementation of the scheme \eqref{galerkinform2A}.   Setting
 \begin{equation*}
 \begin{split}
 &u_N(x)=\sum_{n=1}^{N-1} \hat u_n {}^+ {\hspace*{-3pt}}J_{n}^{(\mu-1,\mu-1)}(x), \quad f_j=(f, I_+^1 P_j),\;\; 1\le j\le N-1,
 \end{split}
 \end{equation*}
 we find from \eqref{Gjcbiform} and the orthogonality of Legendre polynomials that
 \begin{equation}\label{orthresult}
 \big(D_+^{1-\mu}\, {}^+ {\hspace*{-3pt}}J_{n}^{(\mu-1,\mu-1)}, \, D  I_+^1 P_j\big)=
 \frac{\Gamma(n+2-\mu)}{n!} \frac{2n+1} 2 \delta_{jn}.
 \end{equation}
 Then  we  obtain from  \eqref{galerkinform2A} that
 \begin{equation}\label{orthfinal}
 \hat u_n= \frac{2 (n!) f_n} {(2n+1) \Gamma(n+2-\mu) },\quad 1\le n\le N-1.
\end{equation}
We see that using the  GJFs as basis functions, the matrix of the linear system  is diagonal.

\begin{rem}\label{higherorder}
The above approach can be applied to higher-order FBVPs.  For example, we
consider the FBVP of ``odd'' order: $\nu=3-\mu$ with $\mu\in (0,1):$
\begin{equation}\label{prob12Bc}
\begin{split}
  & D_{+}^{\nu}  u(x)=f(x), \;\;  x\in {\Lambda};\quad I_+^{\mu}u(\pm 1)=(I_+^{\mu}u)'(1)=0.
  \end{split}
 \end{equation}
 To avoid repetition, we just outline the numerical scheme and implementation.  Define the solution and test function spaces
 \begin{equation}\label{approxWNA}
 \begin{split}
  & V_N:=\big\{ \phi=(1-x)^{2-\mu}\psi\,: \psi\in {\mathcal  P_{N-2}} \;\; \text{such that}\;\; I_+^{\mu} \phi(-1)=0\big\},
  \\
  & V_N^*:=\{\psi\in {\mathcal P}_N\,:\,\psi(\pm 1)=\psi'(-1)=0  \}. 
  \end{split}
 \end{equation}
  The GJF-Petrov-Galerkin scheme is to find $u_N\in V_N$ such that
 \begin{equation}\label{GJFschems3rd}
 (D^{2-\mu}_+u_N, D w_N)=-(f,w_N),\quad \forall w_N\in V_N^*.
 \end{equation}
%
  Using \eqref{newbateman} with $\rho=\mu,\alpha=2-\mu$ and $\beta=\mu-1,$  we obtain from \eqref{alphaint2} that
\begin{equation}\label{Gjcbiform3A}
\begin{split}
I_+^\mu {}^+ {\hspace*{-3pt}}J_{n}^{(\mu-2,\mu-1)}(x)
= \frac{\Gamma(n+3-\mu)}{(n+2)!} {}^+ {\hspace*{-3pt}}J_{n}^{(-2,-1)}(x); \quad  I_+^\mu {}^+ {\hspace*{-3pt}}J_{n}^{(\mu-2,\mu-1)} (-1)=0,\;\;  n\ge 1.
\end{split}
\end{equation}
Hence, we have
 \begin{equation}\label{approxWNspan}
  V_N={\rm span}\big\{{}^+ {\hspace*{-3pt}}J_{n}^{(\mu-2,\mu-1)}\,:\, 1\le n\le N-2 \big\},\;\;\;
  V_N^*={\rm span}\big\{{}^- {\hspace*{-3pt}}J_{n}^{(-1,-2)}\;:\; 1\le n\le N-2 \big\}.
 \end{equation}
 By \eqref{JacobiForm3},
\begin{equation}\label{newformulas}
D^{2-\mu}_+{}^+ {\hspace*{-3pt}}J_{n}^{(\mu-2,\mu-1)}(x)=\frac{\Gamma(n+3-\mu)}{n!} P_n^{(0,1)}(x),\;\;  D{}^- {\hspace*{-3pt}}J_{m}^{(-1,-2)}(x)=(m+2) (1+x)P_m^{(0,1)}(x),
\end{equation}
 so by the orthogonality of the Jacobi polynomials $\{P_n^{(0,1)}\},$ the matrix of the system  \eqref{GJFschems3rd} is diagonal. \qed
\end{rem}

\subsubsection{FBVPs with Dirichlet boundary conditions}
Now, we turn to a more complicated case,  and consider the fractional boundary value problem
of even order $\nu=s+k$ with $s\in (k-1,k)$ and $k\in {\mathbb N}:$
\begin{equation}\label{prob12A}
\begin{split}
  & D_{+}^{\nu}  u(x)=f(x), \;\;  x\in {\Lambda};\quad u^{(l)}(\pm 1)=0,\;\;  l=0, 1,  \cdots, k-1,
  \end{split}
 \end{equation}
 where  $f(x)$ is a given function with regularity to be specified later.

We introduce the solution  and test function spaces:
 \begin{equation}\label{solusps}
 \begin{split}
& U:=\big\{u\in L^2_{\omega^{(-s,-k)}}(\Lambda)\,:\,  D_+^s u\in L^2_{\omega^{(0,s-k)}}(\Lambda) \big\};\\
& V:=\big\{v\in L^2_{\omega^{(-k,-s)}}(\Lambda)\,:\,  D^k v\in L^2_{\omega^{(0,k-s)}}(\Lambda) \big\},
\end{split}
\end{equation}
equipped with the norms
\begin{equation}\label{eqasT}
\|u\|_U=\big(\|u\|_{\omega^{(-s,-k)}}^2+\|D_+^s u\|_{\omega^{(0,s-k)}}^2\big)^{1/2};\;\;\;
\|v\|_V=\big(\|v\|_{\omega^{(-k,-s)}}^2+\|D^k v\|_{\omega^{(0,k-s)}}^2\big)^{1/2}.
\end{equation}
For $u\in U$ and $v\in V$,   we  write
\begin{equation}\label{uexpansionABC}
\begin{split}
&u(x)=\sum_{n=k}^\infty \hat u_n  {}^+{\hspace*{-3pt}}J_{n}^{(-s,-k)}(x)=(1-x)^s(1+x)^k \sum_{n=k}^\infty \tilde  u_n P_{n-k}^{(s,k)}(x);
\\& v(x)=\sum_{n=k}^\infty \hat v_n  {}^-{\hspace*{-3pt}}J_{n}^{(-k,-s)}(x)=(1-x)^k (1+x)^s \sum_{n=k}^\infty \tilde  v_n P_{n-k}^{(k,s)}(x),
\end{split}
\end{equation}
where by \eqref{zerokpadd},  $\tilde u_n=2^{-k} d_n^{k,s} \hat u_n$ and $\tilde v_n=(-1)^k 2^{-k} d_n^{k,s} \hat v_n.$

With the above setup,  we  can build in the homogenous boundary conditions and also perform fractional integration by parts (cf. Lemma \ref{integration}).
Hence, a weak form  of \eqref{prob12A} is to  find $u\in U$ such that
\begin{equation}\label{weakformulaIIeven}
a(u,v):=(D_+^s  u,\, D^k v)=(f,v),\quad \forall\,v\in V.
\end{equation}

Let ${}^+\hspace*{-2pt} {\mathcal F}_N^{(-s,-k)}(\Lambda)$ and ${}^-\hspace*{-2pt} {\mathcal F}_N^{(-k,-s)}(\Lambda)$ be the finite-dimensional spaces as defined in  the previous section.  
Then the GJF-Petrov-Galerkin scheme for \eqref{weakformulaIIeven} is  to  find
 $u_N\in {}^+\hspace*{-2pt} {\mathcal F}_N^{(-s,-k)}(\Lambda)$ such that
\begin{equation}\label{SchemeIIeven}
a(u_N, v_N)=(D_+^s  u_N,\, D^k v_N)=(f,v_N),\quad \forall\, v_N\in {}^-\hspace*{-2pt} {\mathcal F}_N^{(-k,-s)}(\Lambda).
\end{equation}

 We next show the unique solvability of \eqref{weakformulaIIeven}-\eqref{SchemeIIeven} by  verifying the
 Babu{\v{s}}ka-Brezzi inf-sup condition of the involved
 bilinear form. For this purpose, we first  show the following equivalence of the norms.
 \begin{lmm}\label{normeqiv} Let $s\in (k-1,k),$ $k\in {\mathbb N}$, and  $U, V$ be the space defined in \eqref{solusps} and \eqref{eqasT}, respectively.  Then we have
 \begin{equation}\label{equivanorms}
 \begin{split}
 & C_{1,s}\|u\|_U\le  \|D_+^s u\|_{\omega^{(0,s-k)}}\le \|u\|_U,\quad \forall u\in U;\\
 &  C_{2,s}\|v\|_V\le  \|D^k v\|_{\omega^{(0,k-s)}}\le \|v\|_V,\quad \forall v\in V,
 \end{split}
 \end{equation}
 where
 \begin{equation}\label{constC1C2}
 C_{1,s}=\bigg(1+\frac{k!}{\Gamma(k+s+1)\Gamma(s+1)}\bigg)^{- 1/2};\quad  C_{2,s}=\bigg(1+
 \frac{\Gamma(s+1)}{k!\Gamma(k+s+1)}\bigg)^{-1/2}.
 \end{equation}
 \end{lmm}
 \begin{proof}    Given the expansion in \eqref{uexpansionABC}, we derive from \eqref{orthogonalityII} and
 \eqref{D+orthl}  that
\begin{equation}\label{uexpansionBCH}
\|u\|_{{\omega^{(-s,-k)}}}^2=\sum_{n=k}^\infty \gamma_n^{(s,-k)} \big|\hat u_n\big|^2;\quad
\big\| D_+^{s} u  \big\|_{\omega^{(0,s-k)}}^2=\sum_{n=k}^\infty  h_{n,0}^{(s,-k)}
\big|\hat u_{n}\big|^2,
\end{equation}
where  by  \eqref{constanth},
\begin{equation}\label{h0sk}
h_{n,0}^{(s,-k)}=\frac{\Gamma^2(n+s+1)}{(n!)^2}\gamma_n^{(0,s-k)}.
\end{equation}
Therefore,
\begin{equation*}\label{uexpansionBCHD}
\|u\|_{{\omega^{(-s,-k)}}}^2=\sum_{n=k}^\infty  \frac{\gamma_n^{(s,-k)}} {h_{n,0}^{(s,-k)}}   h_{n,0}^{(s,-k)} \big|\hat u_n\big|^2
\le  \frac{\gamma_k^{(s,-k)}} {h_{k,0}^{(s,-k)}}  \big\| D_+^{s} u  \big\|_{\omega^{(0,s-k)}}^2,
\end{equation*}
so by \eqref{co-gamma}, \eqref{eqasT} and \eqref{h0sk},
$$
\|u\|^2_U \le \bigg(1+ \frac{\gamma_k^{(s,-k)}} {h_{k,0}^{(s,-k)}}\bigg)  \big\| D_+^{s} u  \big\|_{\omega^{(0,s-k)}}^2
=\frac 1 {C_{1,s}^2}  \big\| D_+^{s} u  \big\|_{\omega^{(0,s-k)}}^2.
$$
 This yields  the first equivalence relation in  \eqref{equivanorms}.

Next, we find from
\eqref{2rela} and \eqref{derivative-} that
  \begin{equation}\label{derivative-C}
D^k\big\{{}^{-}{\hspace*{-3pt}} J_{n}^{(-k,-s)}(x)\big\}
=\frac{\Gamma(n+s+1)} {\Gamma(n+s-k+1)}{}^{-}{\hspace*{-3pt}} J_{n}^{(0,k-s)}(x),
\end{equation}
so we have from the orthogonality \eqref{basicorth}-\eqref{orthogonalityII} and  \eqref{uexpansionABC}  that
\begin{equation}\label{uexpansionBCHE}
\|v\|_{{\omega^{(-k,-s)}}}^2=\sum_{n=k}^\infty \big|\hat v_n\big|^2 \gamma_n^{(s,-k)};\quad
\big\| D^k  v  \big\|_{\omega^{(0,k-s)}}^2=\sum_{n=k}^\infty  q_{n}^{(s,k)}
\big|\hat v_{n}\big|^2,
\end{equation}
where
\begin{equation}\label{qnsk}
q_{n}^{(s,k)} := \frac{\Gamma^2(n+s+1)} {\Gamma^2(n+s-k+1)}\gamma_n^{(0,s-k)}.
\end{equation}
Working out the constants leads to
\begin{equation}\label{uexpansionBCHF}
\|v\|_{{\omega^{(-k, -s)}}}^2 \le  \frac{\gamma_k^{(s,-k)}} {q_{k}^{(s,k)}}  \big\| D^{k} v  \big\|_{\omega^{(0,k-s)}}^2
\le  \frac{\Gamma(s+1)}{k!\Gamma(k+s+1)} \big\| D^{k} v  \big\|_{\omega^{(0,k-s)}}^2.
\end{equation}
Then by \eqref{eqasT},  the second equivalence  follows immediately.
\end{proof}

With the aid of Lemma \ref{normeqiv}, we can show the well-posedness of the weak form \eqref{weakformulaIIeven} and the Petrov-Galerkin scheme \eqref{SchemeIIeven}.
 \begin{thm}\label{uniqueevenII} Let $f\in L^2_{\omega^{(s,k)}}(\Lambda).$   Then the problem  \eqref{weakformulaIIeven} admits a unique solution $u\in U$, and the scheme \eqref{SchemeIIeven} admits a unique solution $u_N\in {}^+\hspace*{-2pt} {\mathcal F}_N^{(-s,-k)}(\Lambda).$
\end{thm}
\begin{proof}  
 It is clear that we have the continuity of the bilinear form on $U\times V:$
\begin{equation}\label{continuity}
|a(u,v)|\le \|u\|_U\|v\|_V,\quad \forall\, u\in U,\;\; \forall\, v\in V.
\end{equation}
The main task is to  verify  the inf-sup condition, that is,  for any $0\neq u\in U,$
\begin{equation}\label{infsup}
 \sup_{0\neq v\in V}\frac{|a(u,v)|}{\|u\|_U \|v\|_V}\geq \eta:= C_{1,s}C_{2,s},
\end{equation}
where $C_{1,s}$ and $C_{2,s}$ are given in \eqref{constC1C2}.
For this purpose,  we construct   $v_*\in V$ from  the expansion of $u\in U$  in \eqref{uexpansionABC}: 
\begin{equation}\label{newvstar}
v_*(x):= \sum_{n=k}^\infty \hat v_n^*\,   {}^-{\hspace*{-3pt}}J_{n}^{(-k,-s)}(x) \;\;\; {\rm with} \;\;\;
 \hat v_n^*=\frac{\Gamma(n+s-k+1)}{n!} \hat u_n.
\end{equation}
By construction, one verifies by using from the orthogonality \eqref{jcbiorth},     \eqref{uexpansionBCH} and \eqref{uexpansionBCHE} that
\begin{equation}\label{Derivrela}
a(u,v_*)= \big\| D_+^{s} u  \big\|_{\omega^{(0,s-k)}}^2=\big\| D^k  v_*  \big\|_{\omega^{(0,k-s)}}^2.
\end{equation}
Thus, using Lemma \ref{normeqiv}, we infer that  for any $0\not =u\in U,$  there exists $0\not =v_*\in V$ such that
\begin{equation}\label{DerivrelaB}
a(u,v_*)=\big\| D_+^{s} u  \big\|_{\omega^{(0,s-k)}} \big\| D^k  v_*  \big\|_{\omega^{(0,k-s)}}\ge C_{1,s}C_{2,s}\|u\|_U\|v_*\|_V.
\end{equation}
This implies \eqref{infsup}.

It remains to verify  the ``transposed'' inf-sup condition:
\begin{equation}\label{traninf}
\sup_{0\not=u\in U}|a(u,v)|>0,\quad \forall\, 0\not=v\in V.
\end{equation}
It can be shown by a converse process.  In fact,  assuming that $0\not=v_*\in V$ is an arbitrary function, we  construct
$$
u(x)=\sum_{n=k}^\infty \hat u_n  {}^+{\hspace*{-3pt}}J_{n}^{(-s,-k)}(x)\;\;\; {\rm with}\;\;\;
\hat u_n=\frac {n!}{\Gamma(n+s-k+1)} \hat v_n^*.
$$
Then we can derive \eqref{traninf} using  \eqref{Derivrela}.

Finally, if $f\in L^2_{\omega^{(s,k)}}(\Lambda),$ we obtain from the Cauchy-Schwarz inequality  that
$$|(f,v)|\le \|f\|_{\omega^{(k,s)}}\|v\|_{\omega^{(-k,-s)}}\le \|f\|_{\omega^{(k,s)}}\|v\|_V.$$
Therefore, we claim from the Babu{\v{s}}ka-Brezzi theorem (cf. \cite{Bab.Az72})  that the problem \eqref{galformA} has a unique solution.

Note that the  inf-sup condition \eqref{infsup} is also valid for the discrete problem  \eqref{SchemeIIeven}, which therefore admits a unique solution.
\end{proof}

With the help of the  above results, we can follow a standard argument to carry out the error analysis.
 \begin{thm}\label{ThmIIeven} Let $s\in (k-1,k)$ with $k\in {\mathbb N},$ and let $u$ and $u_N$ be  the solutions of  \eqref{weakformulaIIeven} and  \eqref{SchemeIIeven},  respectively. If  $u\in
 U\cap {\mathcal B}_{s,-k}^m(\Lambda)$ with $0\le m\le N,$ then  we have the error estimates:
\begin{equation}\label{thmresultIIeven}
\|u-u_N\|_U  \leq  cN^{-m}\big\|D_+^{s+m}  u\big\|_{\omega^{(m,s-k+m)}}.
\end{equation}
In particular, if  $f^{(m-k)}\in L^2_{\omega^{(m,s-k+m)}}(\Lambda)$ for $m\geq k$, we have
\begin{equation}\label{thmresultIIevenf}
\|u-u_N\|_U  \leq  cN^{-m}\big\| f^{(m-k)}\big\|_{\omega^{(m,s-k+m)}}.
\end{equation}
Here,  $c$ is a positive constant independent $u,N$ and $m.$
\end{thm}
\begin{proof} Thanks to the inf-sup condition derived in the proof of the previous theorem, we have
\begin{equation}\label{estA1}
\|u-u_N\|_U\leq(1+\eta^{-1}) \|u-\phi\|_{U},\quad \forall\, {\phi\in {}^+\hspace*{-2pt} {\mathcal F}_N^{(-s,-k)}(\Lambda)},
\end{equation}
where $\eta$ is the inf-sup constant in \eqref{infsup}.  Let ${}^+\hspace*{-1pt}\pi_N^{(-s,-k)}$ be the orthogonal projection operator as defined in \eqref{pinuexp}-\eqref{projectorII}.   Taking
$\phi= {}^+\hspace*{-1pt}\pi_N^{(-s,-k)}u$ in \eqref{estA1}, we obtain from  Theorem \ref{Th3.1main} and Lemma \ref{normeqiv} that
\begin{equation}\label{estA2}
\begin{split}
\|u-u_N\|_U &\leq (1+\eta^{-1}) \|u-{}^+\hspace*{-1pt}\pi_N^{(-s,-k)}u\|_{U}\\
&\le  (1+\eta^{-1}) (C_{1,s})^{-1}
\big\|D_+^s (u-{}^+\hspace*{-1pt}\pi_N^{(-s,-k)}u)\big\|_{\omega^{(0,s-k)}}\\
&\le cN^{-m}\big\|D_+^{s+m} u\big\|_{\omega^{(m,s-k+m)}}.
\end{split}
\end{equation}
This yields \eqref{thmresultIIeven}.

From the original equation \eqref{prob12A}, we obtain   $D_+^{\nu} u=D_+^{s+k} u=f$,
so \eqref{thmresultIIevenf} follows from  \eqref{thmresultIIeven} immediately.
\end{proof}

\begin{rem}
 By using a similar procedure as above, we can also construct a spectral Petrov-Galerkin method for the odd order FBVP of order  $\nu=s+k$ and $s\in (k,k+1)$ with $k\in {\mathbb N}:$
 \begin{equation}\label{prob12B}
\begin{split}
  & D_{+}^{\nu}  u(x)=f(x), \;\;  x\in {\Lambda};\quad u^{(l)}(\pm 1)=0,\;\;  l=0, 1,  \cdots, k-1;\;\;  u^{(k)}(1)=0,
  \end{split}
 \end{equation}
 and analyze the error as in Theorem \ref{ThmIIeven}. \qed
\end{rem}

\subsection{Numerical results}

In what follows,  we provide   some  numerical results  to illustrate the accuracy of the proposed  GJF-Petrov-Galerkin schemes and to validate our error analysis.  We gives examples for two typical situations, that is,
the source term $f(x)$ is smooth (so the solution  $u(x)$ is singular), and vice verse. We examine the errors  measured in both  $L^2$-norm and   $\|D^s_+(u-u_N)\|$ (called ``fractional norm'' for simplicity,  to be  in accordance with the analysis), which can be computed from the expansion coefficients. 

\subsubsection{Numerical examples for FIVPs}
We first consider the FIVP \eqref{prob11} with  $f(x)=1+x+\cos x$.
Note that the explicit form of the  exact solution is not available, so we compute a reference exact solution by
using the scheme \eqref{galformA} with  large $N.$  

In view of the error estimate  in Theorem \ref{Th3141}, we know that the errors decay exponentially,  if the source term $f$ is  smooth, despite  that the unknown  solution is singular at $x=1$.  Indeed,  we observe from  Fig.  \ref{Fivps} (left)  that all errors decay  exponentially,  which verify our theoretical results that the convergence rate is only determined by the smoothness of the  source term $f$. Indeed,   we also see that the errors in the fractional norm for different $s$ are indistinguishable, which again show that 
the convergence behaviour solely relies on regularity of  $f.$

\begin{figure}[htp!]
\begin{minipage}{0.495\linewidth}
\begin{center}
\includegraphics[scale=0.40]{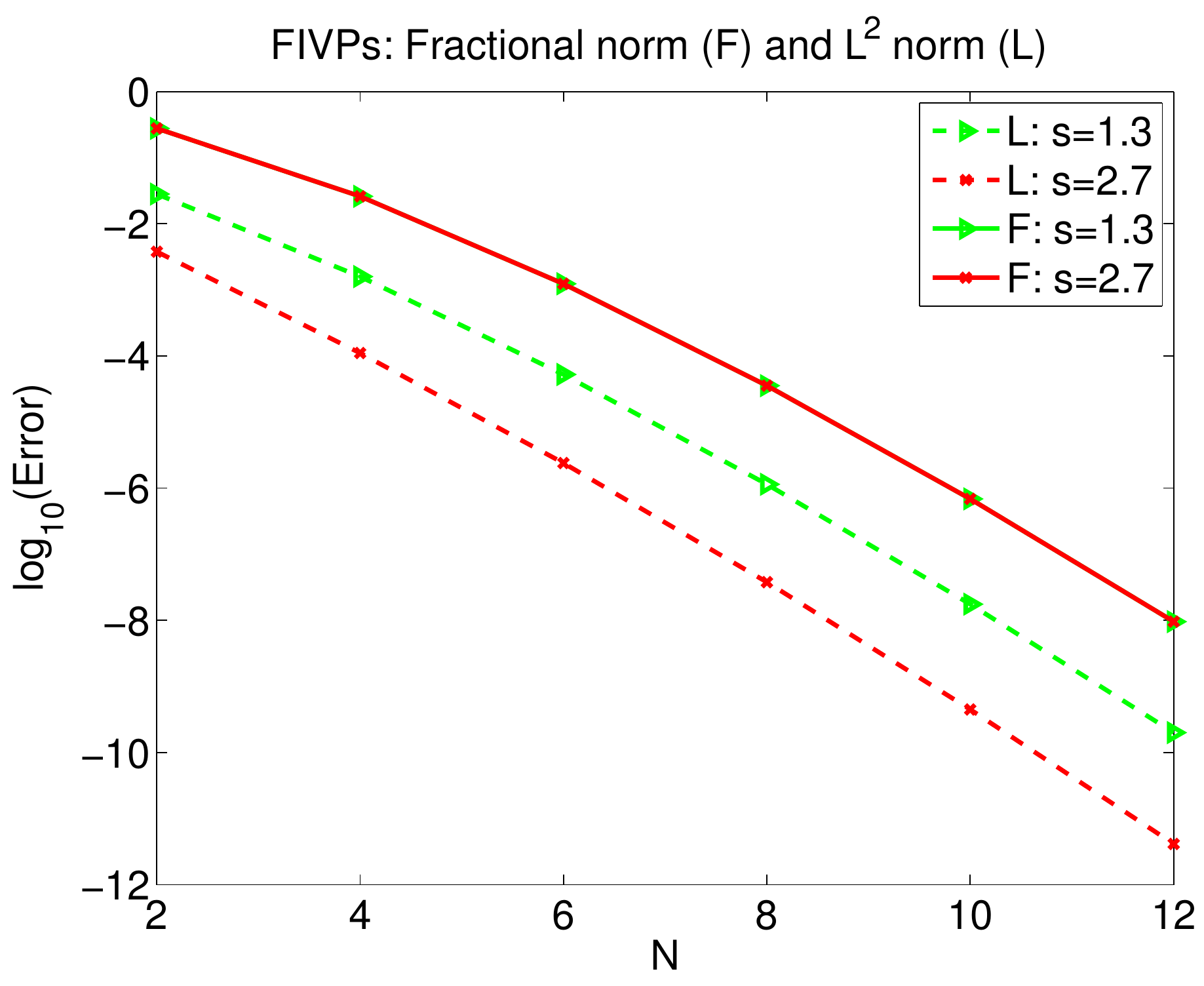}
\end{center}
\end{minipage}
\begin{minipage}{0.495\linewidth}
\begin{center}
\includegraphics[scale=0.40]{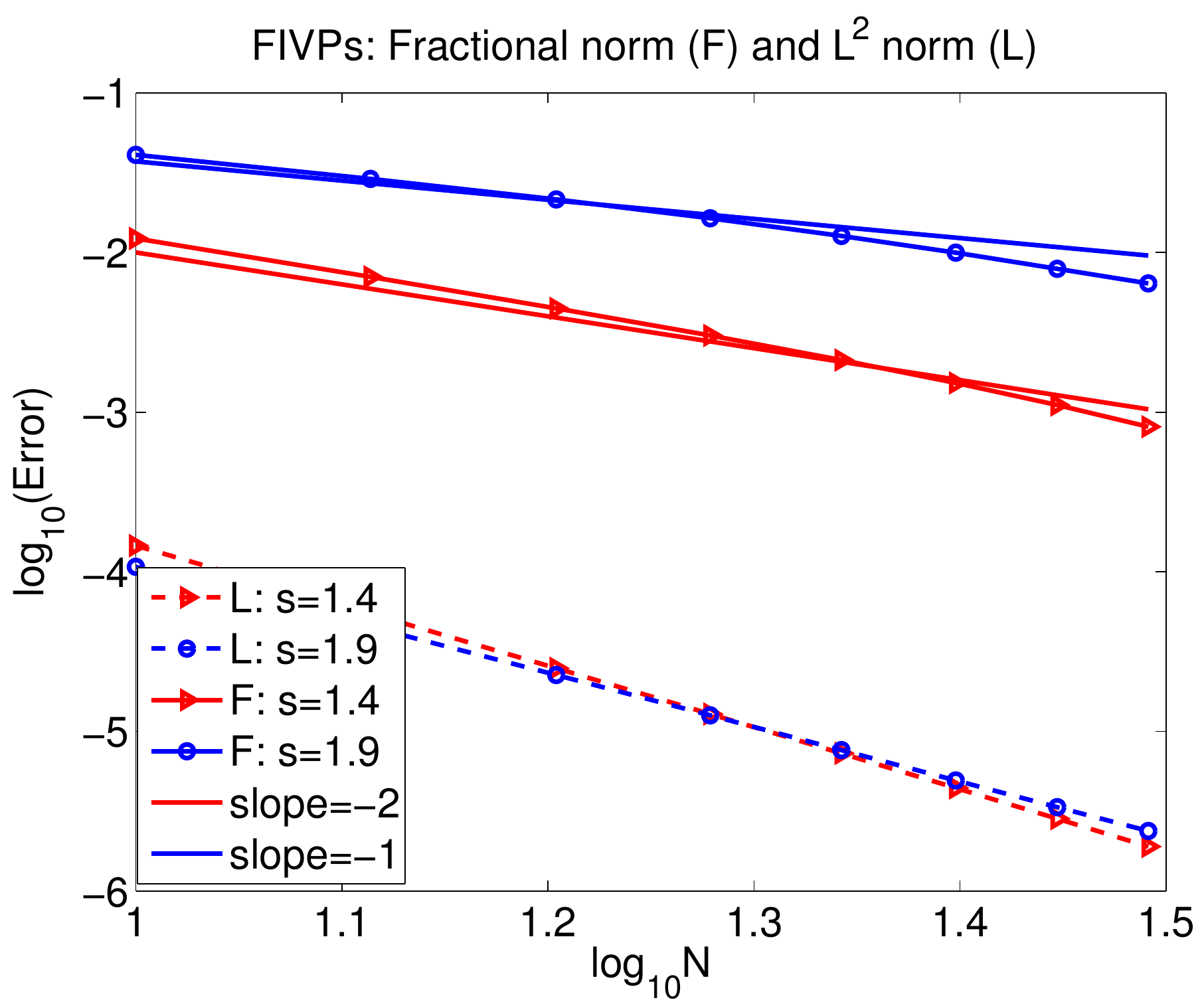}
\end{center}
\end{minipage}
\caption{Convergence of the GJF-Petrov-Galerkin method for the FIVP (\ref{prob11}).  Left:  (\ref{prob11}) with smooth source term $f(x)=1+x+\cos x$. Right:  (\ref{prob11}) with smooth solution: $u(x)=(1-x^3)(1-e^{1-x})$.} \label{Fivps}
\end{figure}

Next, we consider  \eqref{prob11} with $s\in (1,2)$ and the  smooth exact solution:  $u(x)=(1-x^3)(1-e^{1-x})$,  and find the source term $f(x)$ from \eqref{prob11}. It is clear that 
$f(x)$ is singular at $x=1,$ so our error analysis in Theorem \ref{Th3141} predicts  that the convergence rate will be algebraic.  Like in Remark \ref{betterunderstanding},
we calculate  from $u$ that 
$$f^{(m)}=D_+^{s+m}u=O((1-x)^{2-s-m}), \quad x\to 1.$$
Hence, in order to have $\|f^{(m)}\|_{\omega^{(m,m)}}< +\infty$, we need
$$2(2-s-m)+m>-1,\;\;\; \text{i.e.}, \;\;\; m<5-2s,\;\;\;  m\in\mathbb{N}_0.$$
 The convergence behaviours for different $s$ are depicted  in Fig. \ref{Fivps} (right).
 We see that the slopes of the lines agree very well  with the theoretical estimates. 

\subsubsection{Numerical examples for FBVPs with integral boundary conditions}
Now, we consider  the FBVP \eqref{prob12AA} and its GJF-Petrov-Galerkin approximation \eqref{galerkinform2A}. We first take  $f(x)=\sin x$ in \eqref{prob12AA}, and 
compute the reference exact solution as the previous case. 
\begin{figure}[htp!]
\begin{minipage}{0.495\linewidth}
\begin{center}
\includegraphics[scale=0.40]{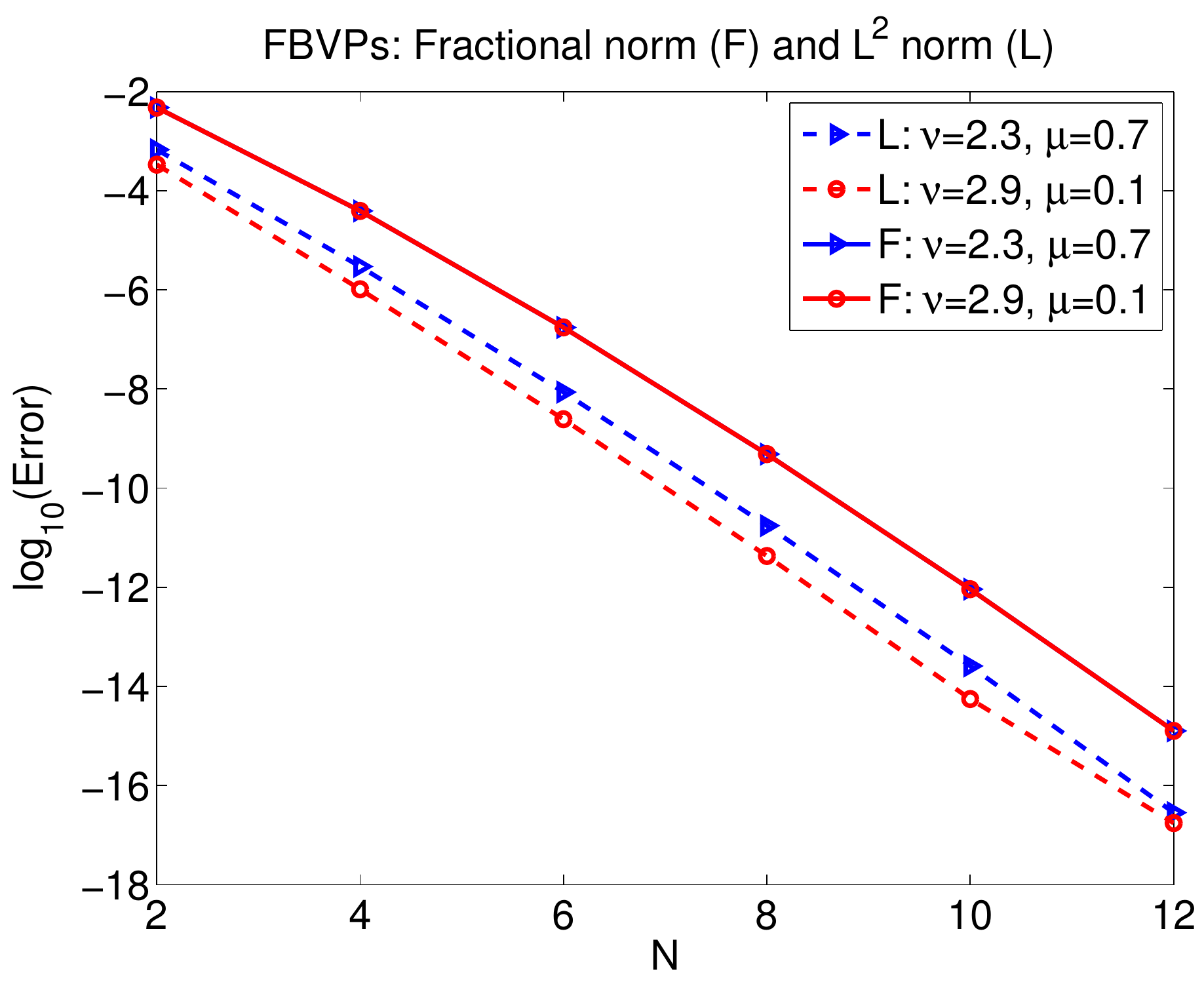}
\end{center}
\end{minipage}
\begin{minipage}{0.495\linewidth}
\begin{center}
\includegraphics[scale=0.40]{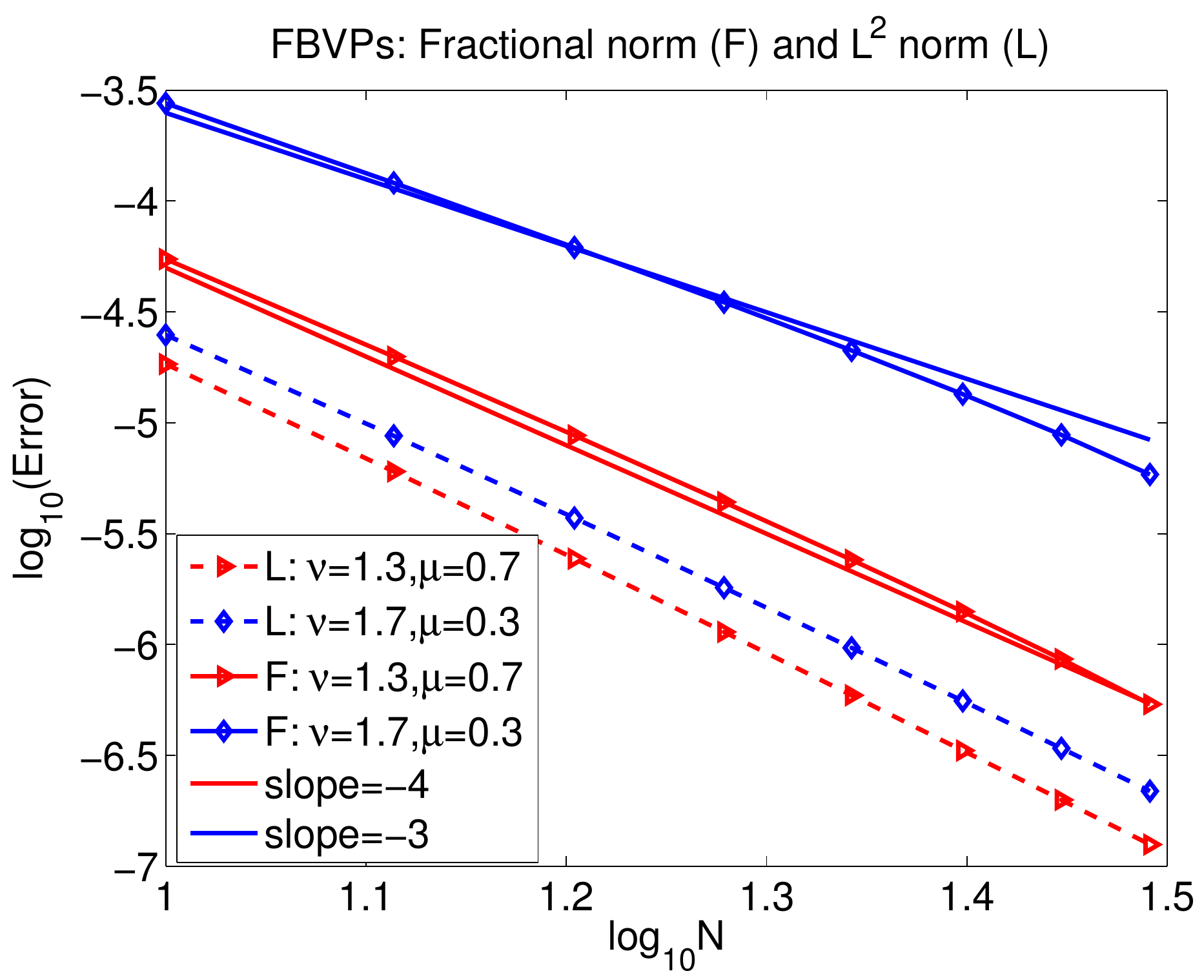}
\end{center}
\end{minipage}
\caption{Convergence of the GJF-Petrov-Galerkin method for the FBVP (\ref{prob12AA}). Left: $f(x)=\sin x$. Right: $u(x)=(1-x)^2(1-x-{6}/{(3+\mu)})$.}\label{FbvpsInt}
\end{figure}
 We  plot the errors for different orders  in   Fig. \ref{FbvpsInt} (left). As expected, the method is truly spectrally  convergent, in agreement with the error estimate \eqref{2ndforderconvergenceC}.  Once again, the convergence rate only depends on the smoothness of $f.$

Next, we take the exact solution to be $u(x)=(1-x)^2(1-x-{6}/{(3+\mu)}), $ and compute $f$ from \eqref{prob12AA}. 
We know from Theorem \ref{convresult} that for $\nu\in(1,2)$ with $\mu=2-\nu$,  we have
 \begin{equation}
\|D^{1-\mu}_+(u-u_N)\|\le c N^{1-m} \|D^{m-\mu}_+u\|_{\omega^{(m-1,m-1)}}.
\end{equation}
A direct calculation shows that, in order for $\|D^{m-\mu}_+u\|_{\omega^{(m-1,m-1)}}<\infty$, we require 
$$2(2-(m-\mu))+m-1>-1,\;\;\; \text{i.e.,}\;\;\; m<4+2\mu,\;\;\; m\in{\mathbb{N}}.$$ Therefore, for $\nu=1.3,\;\mu=0.7$ and $\nu=1.7,\;\mu=0.3$, we have $m<5.4$ and $m<4.6$, respectively, and the expected convergence rate is $m-1$. The numerical errors for this example are plotted in   Fig. \ref{FbvpsInt} (right). We observe that the convergence rates are consistent with our error estimates.

\subsubsection{Numerical examples for FBVPs with homogeneous boundary conditions}
As the last example, we  consider the  FBVP with homogeneous boundary conditions in \eqref{prob12A}.  Similar to the previous cases, we first take a smooth source term $f(x)=xe^x,$ and  plot the errors in  Fig. \ref{Fbvps} (left), which shows  an exponential convergence, as expected from  the error estimates in Theorem \ref{ThmIIeven}.

Next, we take the exact solution $u=(1-x)\sin(\pi x)$ and compute $f$ accordingly from  \eqref{prob12A}.  As before,  we can derive   from the error estimate \eqref{thmresultIIeven}  that the order of convergence $m$ must satisfy  $m<5-2s$ with $m\in {\mathbb N}$. In  Fig. \ref{Fbvps} (right), we plot the errors for
 $\nu=1.4,\;s=0.4$ and $\nu=1.9,\;s=0.9$, respectively.  We again see that  the observed convergence rate agrees  with  the expected rate. 

\begin{figure}[htp!]
\begin{minipage}{0.495\linewidth}
\begin{center}
\includegraphics[scale=0.4]{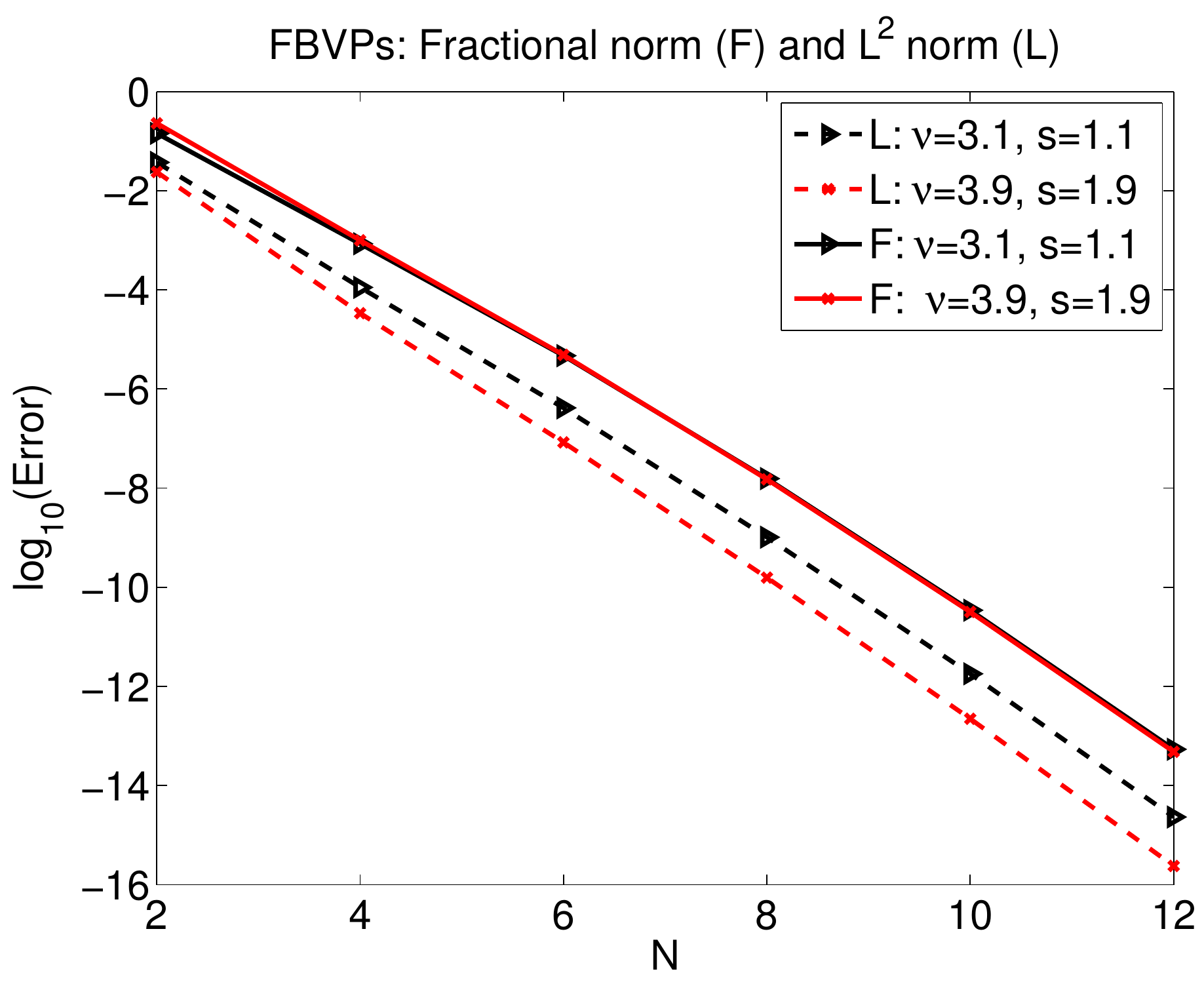}
\end{center}
\end{minipage}
\begin{minipage}{0.495\linewidth}
\begin{center}
\includegraphics[scale=0.4]{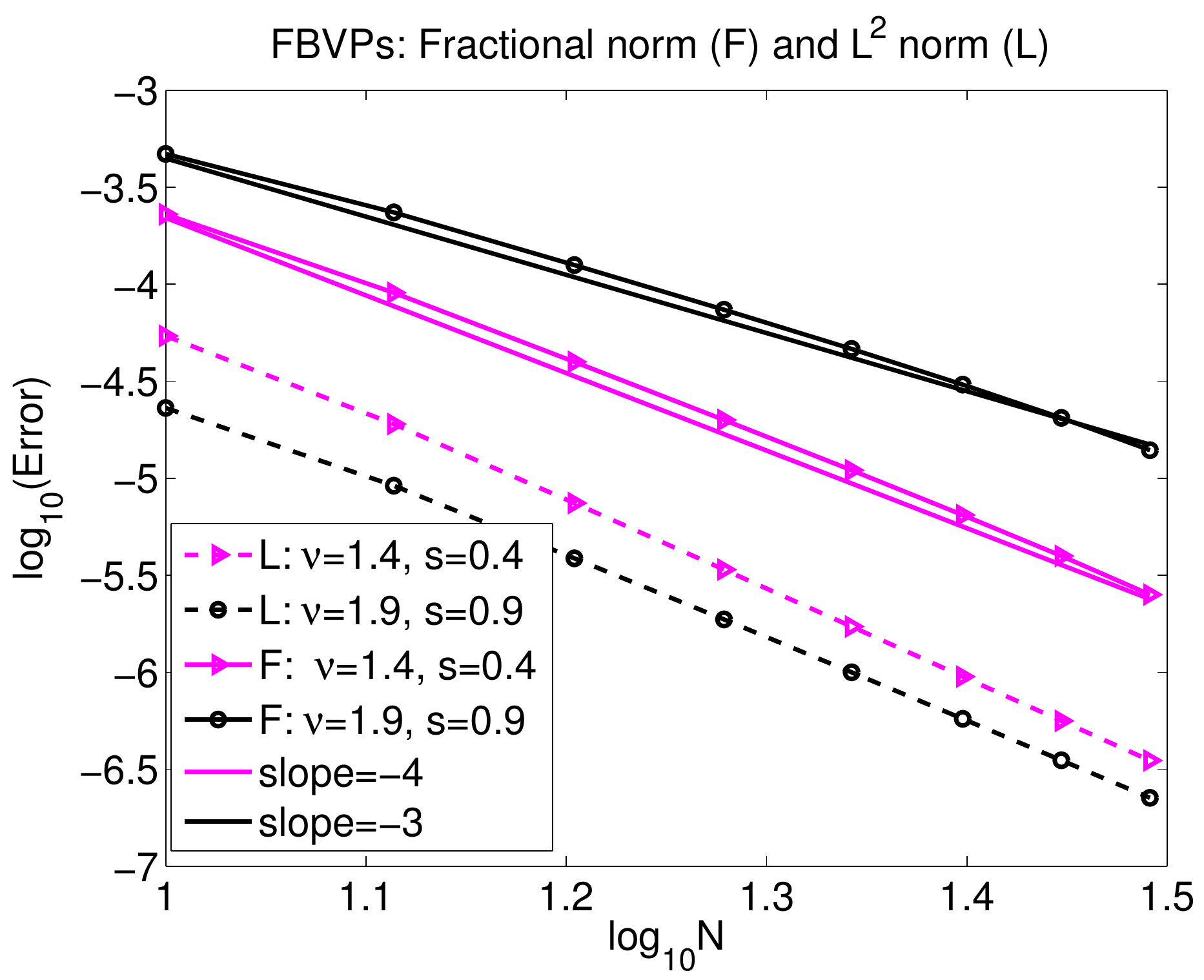}
\end{center}
\end{minipage}
\caption{Convergence of the GJF-Petrov-Galerkin method for the FBVP (\ref{prob12A}).  Left: $f(x)=xe^x$. Right: $u(x)=(1-x)\sin(\pi x)$.}\label{Fbvps}
\end{figure}

Note that in all above examples, the $L^2$ errors are significant smaller than the errors in fractional norms. However, we cannot  justify this rigorously. 
Unlike in the case of integer DEs where one can derive an improved error estimate in the $L^2$-norm using a duality argument,  we are unable to do this in the fractional case largely  due to the lack of regularity in the usual Sobolev norm.  Nevertheless, we see the gain of order in $L^2$-norm from Theorem  \ref{Th3.1main} in the context of approximation by GJFs. 

\section{Extensions, discussions and concluding remarks}\label{sect6app}
\setcounter{equation}{0}
\setcounter{lmm}{0}
\setcounter{thm}{0}

To conclude the paper,  we show that the important formulas of Riemann-Liouville fractional derivatives can be extended in parallel to Caputo derivatives. Consequently, the analysis and results can be generalised to Caputo cases, and  the GJFs enjoy similar remarkable approximability to Caputo FDEs. 
We also provide a summary of main contributions of the paper in the end of this section. 

\subsection{Extension to Caputo derivatives} It is seen that the formulas in Lemma \ref{JacobiForm3} and Theorem \ref{JacobiFormGJFs}  are exceedingly important in the preceding analysis and spectral algorithms involving Riemann-Liouville derivatives.
Remarkably,  similar results are also available  for the  Caputo derivatives.

Like Lemma \ref{JacobiForm3}, we have the following formulas involving Caputo derivatives.
\begin{lemma}\label{CJacobiForm3} Let  $s\in [k-1,k) $ with $ k\in {\mathbb N}$ and $x\in \Lambda.$
 \begin{itemize}
 \item For $\alpha>-1$ and $\beta\in {\mathbb R},$
\begin{equation}\label{Cnewbateman3}
{^C}\hspace*{-2pt}D_+^s\big\{(1-x)^{\alpha+k} P_n^{(\alpha+k,\beta-k)}(x)\big\}
=\frac{\Gamma(n+\alpha+k+1)} {\Gamma(n+\alpha+k-s+1)}(1-x)^{\alpha+k-s} P_n^{(\alpha+k-s,\beta-k+s)}(x).
\end{equation}
\item For   $ \alpha\in {\mathbb R}$ and $ \beta>-1,$
\begin{equation}\label{Cnewbatemanam3s}
{^C}\hspace*{-2pt}D_-^s\big\{(1+x)^{\beta+k} P_n^{(\alpha-k,\beta+k)}(x)\big\}
=\frac{\Gamma(n+\beta+k+1)} {\Gamma(n+\beta+k-s+1)}(1+x)^{\beta+k-s} P_n^{(\alpha-k+s,\beta+k-s)}(x).
\end{equation}
\end{itemize}
\end{lemma}
\begin{proof}  Let us first derive \eqref{Cnewbateman3}.  In view of  $D_+^k=(-1)^k D^k$  (cf. \eqref{2rela}),  we obtain from \eqref{newbateman3} that
\begin{equation}\label{newbateman3q}
D^k\big\{(1-x)^{\alpha+k} P_n^{(\alpha+k,\beta-k)}(x)\big\}
=(-1)^k \frac{\Gamma(n+\alpha+k+1)} {\Gamma(n+\alpha+1)}(1-x)^\alpha P_n^{(\alpha,\beta)}(x).
\end{equation}
By Definition \ref{RLFDdefn}, we have  ${^C}\hspace*{-2pt}D_+^s v=(-1)^k I_+^{k-s}(D^k v),$ so using \eqref{newbateman} with $\rho=k-s$ and \eqref{newbateman3q} leads to
\begin{equation*}\label{Cnewbateman3s}
\begin{split}
{^C}\hspace*{-2pt}D_+^s\big\{(1-x)^{\alpha+k} P_n^{(\alpha+k,\beta-k)}(x)\big\}
& =\frac{\Gamma(n+\alpha+k+1)} {\Gamma(n+\alpha+1)}I_+^{k-s}\big\{(1-x)^\alpha P_n^{(\alpha,\beta)}(x)\big\}\\
&=\frac{\Gamma(n+\alpha+k+1)} {\Gamma(n+\alpha+k-s+1)}(1-x)^{\alpha+k-s} P_n^{(\alpha+k-s,\beta-k+s)}(x).
\end{split}
\end{equation*}
This  yields \eqref{Cnewbateman3}. The formula \eqref{Cnewbatemanam3s} can be derived similarly.
\end{proof}

The counterpart of  Theorem \ref{JacobiFormGJFs} takes a slightly different form in the range of parameters.
\begin{thm}\label{CJacobiFormGJFs} Let  $s\in [k-1,k) $ with $ k\in {\mathbb N}$ and $x\in \Lambda.$
 \begin{itemize}
 \item For $\alpha>k-1$ and $\beta\in {\mathbb R},$
\begin{equation}\label{Cderivative+}
{^C}\hspace*{-2pt}D_+^s\big\{{}^+{\hspace*{-3pt}}J_{n}^{(-\alpha,\beta)}(x)\big\}
=\frac{\Gamma(n+\alpha+1)} {\Gamma(n+\alpha-s+1)}{}^+{\hspace*{-3pt}}J_{n}^{(-\alpha+s,\beta+s)}(x).
\end{equation}
\item For   $ \alpha\in {\mathbb R}$ and $ \beta>k-1,$
\begin{equation}\label{Cderivative-}
{^C}\hspace*{-2pt}D_-^s \big\{{}^{-}{\hspace*{-3pt}} J_{n}^{(\alpha,-\beta)}(x)\big\}
=\frac{\Gamma(n+\beta+1)} {\Gamma(n+\beta-s+1)}{}^{-}{\hspace*{-3pt}} J_{n}^{(\alpha+s,-\beta+s)}(x).
\end{equation}
\end{itemize}
\end{thm}
\begin{proof}   With  $(\alpha-k,\beta+k)$ in place of  $(\alpha,\beta)$ in \eqref{Cnewbateman3}, we obtain \eqref{Cderivative+} immediately from  the definition  \eqref{GJFs+}.
The rule \eqref{Cderivative-} can be obtained in the same fashion.
\end{proof}

Taking  $s=\alpha$ in  \eqref{Cderivative+}  leads to that  for  $\alpha>0$ and $\beta\in {\mathbb R},$
\begin{equation}\label{Cderivative+A}
{^C}\hspace*{-2pt}D_+^s \big\{{}^+{\hspace*{-3pt}}J_{n}^{(-\alpha,\beta)}(x)\big\}=\frac{\Gamma(n+\alpha+1)} {n!}P_n^{(0,\alpha+\beta)}(x).
\end{equation}
Similarly, we derive from \eqref{Cderivative-} an important formula, that is, for $\alpha\in {\mathbb R}$ and real $\beta>0,$
\begin{equation}\label{Cderivative+B}
D_-^\beta\big\{{}^-{\hspace*{-3pt}}J_{n}^{(\alpha,-\beta)}(x)\big\}
=\frac{\Gamma(n+\beta+1)} {n!}P_n^{(\alpha+\beta,0)}(x).
\end{equation}
Indeed, the GJFs with parameter $\alpha>0$ or $\beta>0$  meets the conditions in  \eqref{RL=Ca}, so  we have  the same formulas as  in
\eqref{derivative+A}-\eqref{derivative+B} for the Riemann-Liouville derivatives.

With the aid of the above derivative formulas,  we can establish the GJF approximations in weighted Sobolev spaces,  and
develop efficient spectral methods for FDEs involving Caputo fractional derivatives accordingly.  Here, we omit the details.

\subsection{Discussions and concluding remarks}
We considered in this paper spectral approximation of FDEs by introducing a class of priorly defined GJFs.

Our main contributions are twofold:
\begin{itemize}
 \item Introduced a new class of GJFs, which extend the range of definition of
 polyfractomials \cite{zayernouri2013fractional} so that high-order fractional derivatives can be treated, revealed their relations with fractional derivatives, and studied their approximation properties.
 \item Constructed Petrov-Galerkin spectral methods for a class of prototypical FDEs, including arbitrarily high-order FIVPs and FBVPs which have not been numerically studied before,  which led to  sparse matrices, and
 derived error estimates with convergence rate only depending on the smoothness of  data. In particular, if the data function is analytic, we obtain exponential convergence, despite the fact that the solution is singular.
\end{itemize}

The results presented in this paper indicate that, at least for the simple FDEs considered here, one can develop spectral methods to solve them with the same kind of computational complexity and accuracy as one solve for usual PDEs.

 This is first but important step towards developing efficient and accurate
 spectral methods for solving FDEs. While we have only considered a class of very simple prototypical FDEs, the general principles  and the approximation results developed in this paper open up new possibilities for dealing with more general FDEs.

\end{document}